\documentclass[11pt,twoside,reqno]{amsart}
\usepackage{amsmath}

\usepackage[colorlinks,
linkcolor=red,
citecolor=blue
]{hyperref}

\numberwithin{equation}{section}
\usepackage{appendix}
\usepackage{xcolor}

\def\cX{{\mathcal X}}
\newtheorem{thm}{Theorem}[section]
\newtheorem{lem}{Lemma}[section]

\newtheorem{prop}{Proposition}[section]
\newtheorem{rem}{Remark}[section]

\theoremstyle{definition}
\newtheorem{defn}{Definition}[section]
\usepackage{amsmath,amsfonts,amssymb}
\usepackage{graphicx}    
\usepackage{caption}
\usepackage{subfigure}
\usepackage{cite}
\usepackage{color}
\usepackage{graphicx}
\usepackage{subfigure}
\usepackage{float}
\usepackage{paralist}
\usepackage{indentfirst}
\usepackage{bm}
\usepackage{enumerate}

\def\ddj{\dot \Delta_j}

\def\hat{\widehat}

\newcommand\C{\mathbb{C}}
\newcommand\R{\mathbb{R}}

\newcommand\Z{\mathbb{Z}}

\newcommand{\N}{\mathbb{N}}

\newcommand{\myref}[1]{}

\def\cC{{\mathcal C}}
\def\cD{{\mathcal D}}
\def\cE{{\mathcal E}}

\def\cO{{\mathcal O}}

\def\cX{{\mathcal X}}

\def\dB{{\dot {B}}}

\def\tL{\widetilde{L}}

\newcommand{\Sum}{\displaystyle \sum}

\begin{document}

\title[the Euler-Fourier system with damping system]{
Optimal time-decay for Euler-Fourier system with damping in the critical $L^{2}$ framework}

\author{Jing Liu, Lianchao Gu}



\begin{abstract}
This paper is concerned with the large time behavior of solutions to the Euler-Fourier system with damping in $\R^{d}~(d\geq1)$. A time-weighted energy argument has been developed within the $L^{2}$ framework to derive the optimal time-decay rates, which enables us to remove the smallness of low-frequencies of initial data. A great part of our analysis relies on the study of a Lyapunov functional in the spirit of \cite{LiShou2023}, which mainly depends on some elaborate use of non-classical Besov product estimates and interpolations. Exhibiting a damped mode with faster time decay than the whole solution also plays a key role.

\end{abstract}

\maketitle

\section{Introduction}
In this paper, we consider the following compressible Euler system with damping and heat conduction in $\R^+\times\R^{d}~(d\geq1)$:
\begin{equation}\label{1.1}
\left\{
\begin{array}{l}\rho_t+ \text{div}(\rho u)=0,\\ [1mm]
 (\rho u)_t +\text{div}(\rho u \otimes u) + \nabla P =-\alpha\rho u,\\[1mm]
 (\rho\cE)_t +  \text{div}(\rho u \cE + uP)=-\alpha\rho u^2 + k\Delta  T,
 \end{array} \right.
\end{equation}
with the initial data
\begin{align}\label{1.2}
(\rho, u, T)(x,0)=(\rho_0, u_0, T_0)(x)\rightarrow(\bar{\rho}, 0, \bar{T}),\quad\quad|x|\rightarrow\infty,
\end{align}
where $\bar{\rho}>0$ and $\bar{T}>0$ are some given constants. Here $\rho=\rho(x,t),$ $u=u(x,t)\in\R^{d},$  $T=T(x,t)$ and $P=P(x,t)$ denote the mass density, velocity, absolute temperature and pressure function of the fluid respectively. The total energy $\cE=e+\frac{|u|^2}{2},$ where $e$ is the internal energy. The constant $\alpha>0$ models friction. $k>0$ is the heat conductivity coefficient. 

In this paper, we will consider ideal, polytropic fluids, so that the equations of state for the fluids are given by 
$$P=R\rho T, \ \ \  e=c_v T,$$
where $R>0$ and $c_v>0$ are the universal gas constant and the specific heat at constant volume, respectively. Thus \eqref{1.1} can reduce to the following system:
\begin{equation}\label{1.1**}
\left\{
\begin{array}{l}\rho_t+ \text{div}(\rho u)=0,\\[1mm]
 \rho u_{t} +\rho u\cdot \nabla u + \nabla P =-\alpha\rho u,\\[1mm]
 \rho T_{t}  +  \rho u\cdot \nabla T+\frac{1}{c_v}P\text{div}u=\frac{k}{c_v}\Delta  T,\\[1mm]
 (\rho, u, T)(x,0)=(\rho_0, u_0, T_0)(x).
 \end{array} \right.
\end{equation}

In this paper, we focus on the Cauchy problem \eqref{1.1**} and investigate the asymptotic behavior of global solutions. Let us review some works related to the subject of this paper. In the one-dimensional space case, Nishida \cite{Nishida1,Nishida2} firstly obtained the global existence of a smooth solution with small data and \cite{Hsiao2,Hsiao3,Nishida1,Nishida2} for the large-time behavior of the solution, and the references therein. In multi-dimension space case, the global existence of the small smooth solutions was
proved by different methods in \cite{Sideris,WangYang1} and the large-time behavior of the solution was studied in \cite{Sideris,Tan-2012,WangYang1} and the references therein. For the non-isentropic flow, the global existence of
small smooth solutions to the Cauchy problem was proved in \cite{HsiaoLuo1} in one dimension, and the large time behavior of these solutions can be seen in \cite{HsiaoSerre1,MarcatiPan1}. For multi-dimensions, Tan, Wu and
Huang \cite{WuTanHuang1} proved the global existence of the small smooth solutions and showed that under the
additional assumption that the initial data are bounded in the $L^1$ space. Subsequently, Wu and Miao \cite{WuMiao1} obtained the global existence and large time behavior of the solution when the initial data is close to its equilibrium in $H^3$-norm. In contrast to \cite{WuTanHuang1}, this work removed additional $L^1$ assumptions for the initial value.


In this paper, our aim of this paper is to prove (more accurate) decay estimates in the critical regularity Besov spaces, which is close to the framework of weak solutions. Xu and Kawashima \cite{Xu1, Xu2} investigated partially dissipative hyperbolic system for balance laws, including Euler system with damping, and established the global existence and optimal decay estimates of classical solutions in critical Besov spaces under the Shizuta-Kawashima condition (see \cite{ShizutaKawashima1}).
Recently, Crin - Barat and Danchin \cite{BaratDanchin1, BaratDanchin2, BaratDanchin3} made improvements to the results presented in \cite{Xu1, Xu2}. They proved the global existence and uniqueness of solutions on the $L^2-L^p$ type hybrid homogeneous Besov spaces with different regularity exponents in low and high frequencies and derive the optimal time-decay rates of global solutions. In their study, a specific damped mode emerged, which is a key finding of their research and exhibits a faster time decay compared to the whole solution.

The rest of this paper is organized as follows. In Section \ref{S2}, we state the main results and explain the strategies of this paper. In Section \ref{Low-High}, we establish the low-frequency and high-frequency analysis for the Cauchy problem \eqref{2.2}. In section \ref{S2.2}, we carry out the proofs of Theorems \ref{decay-2} on the optimal time-decay rates of the global solution. In the appendix, we present some notations of Besov spaces and recall related analysis tools used in this paper.\\

\section{Reformulation and main results}\label{S2}
Without loss of generality, set
$$R = c_v=\alpha = k =1,$$
Define the perturbation
$$a:=\rho-1, \ \ a_0:=\rho_0-1,\ \ \theta:=T-1,\ \ \theta_0:=T_0-1.$$
Then, the Cauchy problem \eqref{1.1}-\eqref{1.2} can be reformulated into
\begin{equation}\label{2.2}
\left\{
\begin{array}{l}\partial_{t}a+ u\cdot\nabla a+ (1+a)\text{div}u=0,\\ [1mm]
 \partial_{t}u + \nabla a + u + \nabla \theta + u\cdot \nabla u + \frac{\theta-a}{a+1}\nabla a=0,\\[1mm]
 \partial_t\theta +(1+\theta)\text{div} u+u\cdot\nabla \theta -\frac{1}{a+1}\Delta \theta=0,\\[1mm]
 (a, u,\theta)(x,0)=( a_{0}, u_{0},\theta_{0})(x)\rightarrow(0,0,0),\quad|x|\rightarrow\infty.\\[1mm]
 \end{array} \right.
\end{equation}


In this paper, we focus on the Cauchy problem \eqref{2.2} and investigate the asymptotic behavior of global solution. First of all, let us take a look at the spectral analysis briefly, which was made by Chen, Tan and Wu \cite{ChenTanWu2015}. It is not difficult to check that the linearized system (around the equilibrium) reads
\begin{equation*}
\left\{
\begin{array}{l}\partial_{t}a+ \text{div}u=0,\\[1mm]
 \partial_{t}u + \nabla a + u + \nabla \theta =0,\\[1mm]
 \partial_t\theta +\text{div} u -\Delta \theta=0\\[1mm]
 \end{array} \right.
\end{equation*}
with $a=\rho-1$, $\theta=T-1$. It follows from \cite{ChenTanWu2015} that
\begin{equation*}
\hat{a}(t,\xi)\sim\left\{
\begin{array}{l}C(e^{-c_+|\xi|^{2}t}|\hat{a}_0|+|\xi|e^{-\frac{1}{2}c_+|\xi|^{2}t}|\hat{u}_0|+e^{-\frac{1}{2}c_+|\xi|^{2}t}|\hat{\theta}_0|),\quad|\xi|\rightarrow0;\\[1mm]
C(e^{-\tilde{c}t}|\hat{a}_0|+e^{-\tilde{c}t}|\hat{u}_0|+|\xi|^{-1}e^{-\tilde{c}t}|\hat{\theta}_0|),\quad|\xi|\rightarrow\infty,\\[1mm]
 \end{array} \right.
\end{equation*}
\begin{equation*}
\hat{u}(t,\xi)\sim\left\{
\begin{array}{l}C(|\xi|e^{-c_+|\xi|^{2}t}|\hat{a}_0|+e^{-\frac{1}{2}t}|\hat{u}_0|+|\xi|e^{-\frac{1}{2}c_+|\xi|^{2}t}|\hat{\theta}_0|),\quad|\xi|\rightarrow0;\\[1mm]
C(e^{-\tilde{c}t}|\hat{a}_0|+e^{-\tilde{c}t}|\hat{u}_0|+|\xi|^{-1}e^{-\tilde{c}t}|\hat{\theta}_0|),\quad|\xi|\rightarrow\infty,\\[1mm]
 \end{array} \right.
\end{equation*}
and
\begin{equation*}
\hat{\theta}(t,\xi)\sim\left\{
\begin{array}{l}C(e^{-\frac{1}{2}c_+|\xi|^{2}t}|\hat{a}_0|+|\xi|e^{-\frac{1}{2}c_+|\xi|^{2}t}|\hat{u}_0|+e^{-\frac{1}{2}c_+|\xi|^{2}t}|\hat{\theta}_0|),\quad|\xi|\rightarrow0;\\[1mm]
C(|\xi|^{-1}e^{-\tilde{c}t}|\hat{a}_0|+|\xi|^{-1}e^{-\tilde{c}t}|\hat{u}_0|+|\xi|^{-2}e^{-\tilde{c}t}|\hat{\theta}_0|),\quad|\xi|\rightarrow\infty,\\[1mm]
 \end{array} \right.
\end{equation*}
where $c_+, \tilde{c}$ are generic constants. One can observe that $a$ is associated with a parabolic behavior at low frequencies and an exponentially damped behavior at high frequencies; while $u$ is associated with exponentially damped behavior at both low and high frequencies; $\theta$ is associated with a parabolic behavior at both low and high frequencies. If the initial data $\|(a_0,\theta_0)\|_{L^1}+\|u_0\|_{L^{\frac{3}{2}}}<+\infty$, then one can conclude the optimal decay rates:
$$\|(a,\theta)\|_{L^2}\lesssim(1+t)^{-\frac{3}{4}},\quad\|u\|_{L^2}\lesssim(1+t)^{-\frac{5}{4}}.$$

However, there are few results to our knowledge on the global existence and large time behavior of solutions to \eqref{2.2} in spatially critical spaces.
Then, we establish the global existence of the solution to the Cauchy problem \eqref{2.2} for the initial data close to the equilibrium state below. Recently, Gu and Liu \textcolor{red}{***} established the global existence of small strong solution to $\eqref{2.2}$ in the $L^2$ critical hybrid Besov space. For convenience, we state it
as follows (the reader is also referred to \textcolor{red}{**-Gu}).
\begin{thm}\label{GlobalExist}
For any $d\geq1$, there exists a constant $\varepsilon_{0}>0$ such that if the initial data $(a_0,u_0,\theta_0)$ satisify $(a_0,u_0,\theta_0)\in \dB^{\frac{d}{2}}_{2,1}\cap\dB^{\frac{d}{2}+1}_{2,1}$, and
\begin{align}\label{IniCon1}
\cX_0\triangleq \|(a_0,u_0,\theta_0)\|^{\ell}_{\dB^{\frac{d}{2}}_{2,1}}+\|(a_0,u_0,\theta_0)\|^{h}_{\dB^{\frac{d}{2}+1}_{2,1}} \leq \varepsilon_0,
\end{align}
then the Cauchy problem \eqref{2.2} admits a unique global strong solution $(a,u,\theta)$, which satisfies
\begin{equation}\label{1dlg-E4-3.6}
 \left\{
\begin{aligned}
&a^{\ell}\in \cC_{b}(\R^+;\dB^{\frac{d}{2}}_{2,1})\cap \tL^2(\R^+;\dB^{\frac{d}{2}+1}_{2,1}),~a^{h}\in \cC_{b}(\R^+;\dB^{\frac{d}{2}+1}_{2,1})\cap \tL^2(\R^+;\dB^{\frac{d}{2}+1}_{2,1}),\\
&u^{\ell}\in \cC_{b}(\R^+;\dB^{\frac{d}{2}}_{2,1})\cap \tL^2(\R^+;\dB^{\frac{d}{2}}_{2,1}),~u^{h}\in \cC_{b}(\R^+;\dB^{\frac{d}{2}+1}_{2,1})\cap \tL^2(\R^+;\dB^{\frac{d}{2}+1}_{2,1}),\\
&\theta^{\ell}\in \cC_{b}(\R^+;\dB^{\frac{d}{2}}_{2,1})\cap \tL^2(\R^+;\dB^{\frac{d}{2}+1}_{2,1}),~\theta^{h}\in \cC_{b}(\R^+;\dB^{\frac{d}{2}+1}_{2,1})\cap \tL^2(\R^+;\dB^{\frac{d}{2}+2}_{2,1}),
\end{aligned}
\right.
\end{equation}
and
\begin{equation}\label{IniCon}
\begin{split}
\cX&\triangleq\|(a,u,\theta)\|^{\ell}_{\tL^\infty_t(\dB^{\frac{d}{2}}_{2,1})} +\|(a,u,\theta)\|^{h}_{\tL^\infty_t(\dB^{\frac{d}{2}+1}_{2,1})}+(a,\theta)\|^{\ell}_{\tL^2_t(\dB^{\frac{d}{2}+1}_{2,1})}\\
&\quad+\|u\|^{\ell}_{\tL^2_t(\dB^{\frac{d}{2}}_{2,1})}+\|(a,u)\|^{h}_{\tL^2_t(\dB^{\frac{d}{2}+1}_{2,1})} + \|\theta\|^{h}_{\tL^2_t(\dB^{\frac{d}{2}+2}_{2,1})}\\
&\leq C\cX_{0},~~~~t>0,
\end{split}
\end{equation}
for $C>0$ a constant independent of time.
\end{thm}
For Besov spaces, the readers can refer to Definitions \ref{D5.1}-\ref{D5.2} in the appendix.
Next, we study the optimal time-decay rates of global solutions  to system \eqref{2.2}.
\begin{thm}\label{decay-2}
For any $d\geq1$, let $(a,u,\theta)$ be the corresponding global solution of \eqref{2.2}. If the low-frequency part of the initial data $(a_0,u_0,\theta_0)$ additionally fulfills
\begin{equation}\label{lowinitial}
(a_0,u_0,\theta_0)^{\ell}\in\dB^{-\sigma_{1}}_{2,\infty}\quad \mbox{for}\quad \sigma_{1}\in(-\frac{d}{2},\frac{d}{2}],
\end{equation}
then it holds for any $t\in\R^+$ that
\begin{equation}\label{decay}
\|(a,u,\theta)(t)\|_{\dB^{\sigma}_{2,1}}\leq C\delta_{0}(1+t)^{-\frac{1}{2}(\sigma+\sigma_{1})},\quad\sigma\in(-\sigma_{1},\frac{d}{2}],
\end{equation}
with a constant $C>0$ independent of time, and
\begin{equation}\label{decay-1}
\delta_{0}\triangleq\|(a_0,u_0,\theta_0)\|^{\ell}_{\dB^{-\sigma_{1}}_{2,\infty}}+\|(a_0,u_0,\theta_0)\|^{h}_{\dB^{\frac{d}{2}+1}_{2,1}}.
\end{equation}
Furthermore, if $d\geq2$ and $\sigma_1\in(-\frac{d}{2}+1,\frac{d}{2}]$, then the following time-decay estimate holds:
\begin{equation}\label{decay-3}
 \left\{
\begin{aligned}
&\|u\|_{\dB^{-\sigma_{1}}_{2,\infty}}\leq C\delta_{0}(1+t)^{-\frac{1}{2}},\\
&\|u\|_{\dB^{\sigma}_{2,1}}\leq C\delta_{0}(1+t)^{-\frac{1}{2}(1+\sigma+\sigma_{1})},\quad\sigma\in(-\sigma_{1},\frac{d}{2}-1].
\end{aligned}
\right.
\end{equation}
\end{thm}

\begin{rem}
More precisely, one can obtain the classical decay rates. For instance, if taking $\sigma=\sigma_1=\frac{3}{2}$, then the solution decays in the spatially critical space $L^{\infty}$ at the rate of 
$\cO(t^{-\frac{3}{2}})$. If taking $\sigma=0,~\sigma_1=\frac{3}{2}$, then the solution decays in $L^2$ at the rate of
$\cO(t^{-\frac{3}{4}})$ and the damping term $u$ in the velocity equation is at the enhanced rate of $\cO(t^{-\frac{5}{4}})$, which were shown by Chen, Tan and Wu {\rm\cite{ChenTanWu2015}} in the framework of Sobolev spaces with higher regularity.
\end{rem}

To show Theorem \ref{decay-2}, our approach is based on new time-weighted energy estimates and interpolation arguments, which is an adaptation of the previous approaches in \cite{LiShou2023, Xin-2021} and enable us to derive the optimal time-decay rates of the solution $(a,u,\theta)$ in \eqref{decay} (refer
to Lemmas \ref{LowEsta}-\ref{LowEsta-1}). Furthermore, in order to improve our low frequency analysis, we will exhibit a damped mode with better decay properties than the whole solution. Our approach will enable us to remove the additional requirement of smallness of low frequencies.\\

\textbf{Notation.} Throughout the paper, we denote by C harmless positive constants that may change from one line to the other. We sometimes write $A\lesssim B$ instead of $A\leq CB$. Likewise, $A\thicksim B$ means that $C_1B\leq A \leq C_2B$ with absolute constants $C_1, C_2$. For any Banach space $X$ and $f, g\in X$, we agree that $\|(f,g)\|_{X}\triangleq\|f\|_{X}+\|g\|_{X}$. For all $T>0$ and $\varrho\in [1,\infty]$, we denote by $L^{\varrho}_{T}(X)\triangleq L^{\varrho}([0,T];X)$ the set of measurable function $f:[0,T]\rightarrow X$ such that $t\mapsto\|f(t)\|_X$ is in $L^{\varrho}(0,T)$.


\section{Low-frequency and High-frequency analysis}\label{Low-High}
We are going to establish a Lyapunov-type inequality for energy norms by using a pure energy argument. For clarity, the proof is divided into several steps. 
\subsection{Low-frequency analysis} In this subsection, we first establish the a priori estimates of solutions to the Cauchy problem (\ref{2.2}) in the low-frequency region $\{\xi\in\mathbb{R}^{3}\mid|\xi|\leq\frac{8}{3}\}$. Note that \eqref{1dlg-E4-3.6} implies
\begin{equation}\label{bounded}
\begin{split}
&\sup\limits_{(x,t)\in\R^{d}\times(0,T)}|a(x,t)|\leq\frac{1}{2}\Rightarrow\frac{1}{2}\leq\varrho=1+a\leq\frac{3}{2},\\
&\sup\limits_{(x,t)\in\R^{d}\times(0,T)}|\theta(x,t)|\leq\frac{1}{2}\Rightarrow\frac{1}{2}\leq T=1+\theta\leq\frac{3}{2},\quad\mbox{if}~\cX_{0}\ll1.
\end{split}
\end{equation}
The property \eqref{bounded} will be used to handle the nonlinear terms $\frac{\theta-a}{a+1}$ and $\frac{1}{1+a}$
in \eqref{2.2} by the continuity of
composition estimates in Proposition \ref{ParalinearizationTheorem}. For any $j\in\Z$, applying the operator $\ddj$ to \eqref{2.2}, we get
\begin{equation}\label{LocalizationEq}
\left\{
\begin{array}{l}\partial_{t}\ddj a+\ddj\text{div} u + \ddj\text{div}(au)=0,\\ [1mm]
\partial_{t}\ddj u + \ddj\nabla a + \ddj u + \ddj\nabla\theta + \ddj(u\cdot\nabla u) + \ddj(\frac{\theta-a}{a+1}\nabla a)=0,\\[1mm]
\partial_{t}\ddj\theta -\ddj\Delta\theta +\ddj\text{div}u + \ddj\text{div}(u\theta) +\ddj(\frac{a}{a+1}\Delta\theta)=0.\\[1mm]
 \end{array} \right.
\end{equation}

First, we derive a low-frequency Lyapunov type inequality of \eqref{LocalizationEq}.
\begin{lem}\label{LowL}
Let $(a,u,\theta)$ be any strong solution to the Cauchy problem \eqref{2.2}. Then, it holds for any $j\leq0$ that
\begin{eqnarray}\label{ELowL}
\begin{aligned}
&\frac{d}{dt}\|(\ddj a,\ddj u, \ddj\theta)\|^2_{L^2} + 2^{2j}\|(\ddj a, \ddj\theta)\|^2_{L^2}+\|\ddj u\|^2_{L^2}\\
&\quad\lesssim
\|(\ddj(au), \ddj(u\cdot\nabla u),\ddj(\frac{\theta-a}{a+1}\nabla a))\|_{L^2}\|(\ddj u, \ddj\nabla a)\|_{L^2}  \\
&\quad\quad+ \|(\ddj(u\theta),\ddj\big(\frac{a}{a+1}\nabla\theta\big))\|_{L^2}\|\ddj\nabla\theta\|_{L^2}+\|\ddj\big(\nabla(\frac{a}{a+1})\nabla\theta\big)\|_{L^2}\|\ddj\theta\|_{L^2}.
\end{aligned}
\end{eqnarray}
\end{lem}
\begin{proof}
Taking the $L^2$ inner product of $\eqref{LocalizationEq}_{1}$, $\eqref{LocalizationEq}_{2}$ and $\eqref{LocalizationEq}_{3}$ with $\ddj a$, $\ddj u$ and $\ddj \theta$, respectively, we have
\begin{equation}\label{LocalizationEq1}
\frac{1}{2}\frac{d}{dt}\|\ddj a\|^2_{L^2}=-\int_{\R^d}\ddj\text{div} u\ddj adx-\int_{\R^d}\ddj\text{div}(au)\ddj adx,
\end{equation}
\begin{equation}\label{LocalizationEq2}
\begin{split}
\frac{1}{2}\frac{d}{dt}\|\ddj u\|^2_{L^2}+\|\ddj u\|^2_{L^2}&=-\int_{\R^d}\ddj\nabla a\ddj udx-\int_{\R^d}\ddj\nabla\theta\ddj udx\\
&\quad-\int_{\R^d}\ddj(u\cdot\nabla u)\ddj udx-\int_{\R^d}\ddj(\frac{\theta-a}{a+1}\nabla a)\ddj udx,
\end{split}
\end{equation}
and
\begin{equation}\label{LocalizationEq3}
\begin{split}
\frac{1}{2}\frac{d}{dt}\|\ddj \theta\|^2_{L^2}+\|\nabla\theta\|^2_{L^2}&=-\int_{\R^d}\ddj\text{div}(u\theta)\ddj \theta dx-\int_{\R^d}\ddj\text{div}u\ddj \theta dx\\
&\quad-\int_{\R^d}\ddj(\frac{a}{a+1}\Delta\theta)\ddj \theta dx.
\end{split}
\end{equation}
By performing an integration
by parts in the third term on the right-hand side of \eqref{LocalizationEq3}, we get
\begin{align*}
&\int_{\R^d}\ddj(\frac{a}{a+1}\Delta\theta)\ddj\theta dx\\
&=\int_{\R^d}\ddj\text{div}\big(\frac{a}{a+1}\nabla\theta\big)\ddj\theta dx- \int_{\R^d}\ddj\big(\nabla(\frac{a}{a+1})\nabla\theta\big)\ddj\theta dx\\
&=-\int_{\R^d}\ddj\big(\frac{a}{a+1}\nabla\theta\big)\ddj\nabla\theta dx- \int_{\R^d}\ddj\big(\nabla(\frac{a}{a+1})\nabla\theta\big)\ddj\theta dx.
\end{align*}
The combination of \eqref{LocalizationEq1}-\eqref{LocalizationEq3} leads to
\begin{equation}\label{LocalizationEq4}
\begin{split}
&\frac{1}{2}\frac{d}{dt}\|(\ddj a,\ddj u,\ddj\theta)\|^2_{L^2} + \|\ddj u\|^2_{L^2} + \|\ddj\nabla\theta\|^2_{L^2} \\
&\quad=-\int_{\R^d}\ddj\text{div}(au)\ddj adx- \int_{\R^d}\ddj(u\cdot\nabla u)\ddj udx - \int_{\R^d}\ddj(\frac{\theta-a}{a+1}\nabla a)\ddj udx \\
&\quad\quad- \int_{\R^d}\ddj\text{div}(u\theta)\ddj\theta dx +\int_{\R^d}\ddj\big(\frac{a}{a+1}\nabla\theta\big)\ddj\nabla\theta dx+ \int_{\R^d}\ddj\big(\nabla(\frac{a}{a+1})\nabla\theta\big)\ddj\theta dx.
\end{split}
\end{equation}

In order to obtain the dissipation of $a$, we make use of $\eqref{LocalizationEq}_{1}$ and $\eqref{LocalizationEq}_{2}$ to have
\begin{eqnarray}\label{123456}
\begin{aligned}
&\frac{d}{dt}\int_{\R^d}\ddj \nabla a\ddj udx + \|\ddj \nabla a\|_{L^2}^2-\|\ddj \text{div} u\|_{L^2}^2  + \int_{\R^d}\ddj u\ddj\nabla adx\\
&\quad+ \int_{\R^d}\ddj\nabla\theta\ddj\nabla adx=\int_{\R^d}\ddj\text{div}(au)\ddj\text{div}udx  \\
&\quad- \int_{\R^d}\ddj(u\cdot\nabla u)\ddj\nabla adx 
- \int_{\R^d}\ddj\big(\frac{\theta-a}{a+1}\nabla a\big)\ddj\nabla adx.
\end{aligned}
\end{eqnarray}

Furthermore, adding up \eqref{123456} (multiplying by $\eta_{1}$ for some $\eta_1\in(0,1)$) to \eqref{LocalizationEq4} yields
\begin{eqnarray}\label{1234567}
\begin{aligned}
&\frac{d}{dt}\cE_{1,j}(t) + \cD_{1,j}(t)\\
&\quad\lesssim(1+2^{2j}\eta_{1})
\|\ddj(au)\|_{L^2}\|(\ddj \nabla a, \ddj u)\|_{L^2} \\
&\quad\quad+ (1+\eta_{1})\|(\ddj(u\cdot\nabla u),\ddj(\frac{\theta-a}{a+1}\nabla a))\|_{L^2}\|(\ddj u, \ddj\nabla a)\|_{L^2}  \\
&\quad\quad+ \|(\ddj(u\theta),\ddj\big(\frac{a}{a+1}\nabla\theta\big))\|_{L^2}\|\ddj\nabla\theta\|_{L^2}+\|\ddj\big(\nabla(\frac{a}{a+1})\nabla\theta\big)\|_{L^2}\|\ddj\theta\|_{L^2},
\end{aligned}
\end{eqnarray}
where $\cE_{1,j}(t)$ and $\cD_{1,j}(t)$ are defined by
\begin{eqnarray*}
\left\{
\begin{aligned}
\cE_{1,j}(t)&\triangleq\frac{1}{2}\|(\ddj a,\ddj u,\ddj\theta)\|^2_{L^2} + \eta_1\int_{\R^d}\ddj \nabla a\ddj udx,\\
\cD_{1,j}(t)&\triangleq\|\ddj u\|^2_{L^2} + \|\ddj\nabla\theta\|^2_{L^2}+ \eta_1\|\ddj \nabla a\|_{L^2}^2 - \eta_1\|\ddj\text{div}u\|^2_{L^2}\\[1mm]
&\quad
+\eta_1\int_{\R^d}\ddj u\ddj\nabla adx
+\eta_1\int_{\R^d}\ddj\nabla\theta\ddj\nabla adx
\end{aligned}
\right.
\end{eqnarray*}
with constant $\eta_1 \in(0,1)$ to be determined later. 
One can show for any $j \leq 0$ that
\begin{equation}\label{4*}
(\frac{1}{2}-\frac{\eta_1}{2})\|(\ddj a,\ddj u,\ddj\theta)\|^2_{L^2}\leq \cE_{1,j}(t) \leq(\frac{1}{2}+\frac{\eta_1}{2})\|(\ddj a,\ddj u, \ddj\theta)\|^2_{L^2}
\end{equation}
and
\begin{eqnarray}\label{5*}
\begin{aligned}
\cD_{1,j}(t)\geq C\eta_1 2^{2j}\|\ddj a\|^2_{L^2} + (1-2^{2j}\eta_1-C\eta_1)\|\ddj u\|^2_{L^2} + 2^{2j}(1-C\eta_1)\|\ddj\theta\|^2_{L^2},
\end{aligned}
\end{eqnarray}
where $C>1$ denotes a sufficiently large constant independent of time. Choosing a sufficiently small constant $\eta_{1}\in(0,1)$, we deduce by \eqref{4*}-\eqref{5*} for any $j\leq0$ that
\begin{equation}\label{6*}
\cE_{1,j}(t) \sim\|(\ddj a,\ddj u, \ddj\theta)\|^2_{L^2}
\end{equation}
and
\begin{equation}\label{7*}
\cD_{1,j}(t)\gtrsim 2^{2j}\|(\ddj a, \ddj\theta)\|^2_{L^2}+\|\ddj u\|^2_{L^2}.
\end{equation}
By \eqref{1234567} and \eqref{6*}-\eqref{7*}, \eqref{ELowL} follows.

\end{proof}

\subsection{High-frequency analysis}In this subsection, we establish solutions to the Cauchy problem \eqref{2.2} in the high-frequency region $\{\xi\in\mathbb{R}^{d}\mid|\xi|\geq\frac{3}{8}\}$. To this end, we show a high-frequency Lyapunov type inequality of \eqref{LocalizationEq}.
\begin{lem}\label{HighEsta}
Let $(a,u,\theta)$ be any strong solution to the Cauchy problem \eqref{2.2}. Then, it holds for any $j\geq-1$ that 
\begin{equation}\label{high*}
\begin{aligned}
&\frac{d}{dt}\|(\ddj a,\ddj u, \ddj\theta)\|^2_{L^2} + \|(\ddj a, \ddj u)\|^2_{L^2}+2^{2j}\|\ddj \theta\|^2_{L^2}\\
&\quad\lesssim \Big\|\frac{\partial}{\partial t}\Big(\frac{1+\theta}{(1+a)^2}\Big)\Big\|_{L^\infty}\|\ddj a\|_{L^2}^2 + \Big\|\nabla\Big(\frac{1+\theta}{1+a}\Big)\Big\|_{L^\infty}\|\ddj u\|_{L^2}\|\ddj a\|_{L^2} \\
&\quad\quad+ \Big\|\mathrm{div}\Big(\frac{1+\theta}{(1+a)^2}u\Big)\Big\|_{L^\infty}\|\ddj a\|^2_{L^2} + \|\mathrm{div}u\|_{L^\infty}\|\ddj u\|^2_{L^2}+\|\ddj (u\theta)\|_{L^2}\|\ddj \nabla\theta\|_{L^2} \\[1mm]
&\quad\quad+ \|\nabla(\frac{a}{1+a})\|_{L^\infty}\|\ddj\nabla\theta\|_{L^2}\|\ddj\theta\|_{L^2}+ \|\frac{a}{1+a}\|_{L^\infty}\|\ddj\nabla\theta\|_{L^2}^2 \\[1mm]
&\quad\quad+\|\tilde{R}^1_j\|_{L^2}\|\frac{1+\theta}{(1+a)^2}\|_{L^{\infty}}\|\ddj a\|_{L^2} + \|\tilde{R}^2_j\|_{L^2}\|\ddj u\|_{L^2} + \|\tilde{R}^3_j\|_{L^2}\|\ddj \theta\|_{L^2} \\[1mm]
&\quad\quad+2^{-2j}\|\ddj\mathrm{div}(au)\|_{L^2}\|\ddj\mathrm{div}u\|_{L^2}+2^{-2j}\|\ddj(u\cdot\nabla u)\|_{L^2}\|\ddj\nabla a\|_{L^2}\\
&\quad\quad+2^{-2j}\|\ddj\big(\frac{\theta-a}{a+1}\nabla a\big)\|_{L^2}\|\ddj\nabla a\|_{L^2}
\end{aligned}
\end{equation}
\end{lem}
\begin{proof}
The Cauchy problem \eqref{LocalizationEq} can be reformulated as
\begin{equation}\label{4.14}
\left\{
\begin{array}{l}\partial_{t}\ddj a+ u\ddj\nabla a+ (1+a)\ddj\text{div}u=\tilde{R}^1_j,\\ [1mm]
 \partial_{t}\ddj u  + \ddj u + \ddj \nabla \theta + u\ddj \nabla u + \frac{1+\theta}{1+a}\ddj\nabla a=\tilde{R}^2_j,\\[1mm]
 \partial_t\ddj \theta +\ddj\text{div} u-\ddj\Delta\theta+\ddj\text{div}(u\theta) -\frac{a}{a+1}\ddj\Delta \theta=\tilde{R}^3_j,
 \end{array} \right.
\end{equation}
where
\begin{equation}\label{reminder}
\begin{aligned}
&\tilde{R}^1_j=-[\ddj, (1+a)]\text{div}u -[\ddj, u]\nabla a,\\
&\tilde{R}^2_j=-[\ddj, u] \nabla u - [\ddj, \frac{1+\theta}{1+a}]\nabla a,\\
&\tilde{R}^3_j= [\ddj,\frac{a}{a+1}]\Delta\theta.
\end{aligned}
\end{equation}
By similar arguments to Lemma \ref{LowL}, we obtain
\begin{equation}\label{4.1401}
\begin{split}
&\frac{1}{2}\frac{d}{dt}\int_{\R^d}\frac{1+\theta}{(1+a)^2}(\ddj a)^2dx  + \int_{\R^d}\frac{1+\theta}{1+a}\ddj\text{div}u\ddj adx + \int_{\R^d}\frac{1+\theta}{(1+a)^2}u\ddj\nabla a\ddj adx\\
&=\frac{1}{2}\int_{\R^d}\frac{\partial}{\partial t}\left(\frac{1+\theta}{(1+a)^2}\right)(\ddj a)^2dx+\int_{\R^d}\tilde{R}_j^1\frac{1+\theta}{(1+a)^2}\ddj adx,
\end{split}
\end{equation}
\begin{equation}\label{4.1402}
\begin{aligned}
&\frac{1}{2}\frac{d}{dt}\|\ddj u\|_{L^2}^2 +\|\ddj u\|_{L^2}^2 + \int_{\R^d}u\ddj\nabla u\ddj udx \\
&+\int_{\R^d}\ddj\nabla \theta\ddj udx + \int_{\R^d}\frac{1+\theta}{1+a}\ddj\nabla a\ddj udx= \int_{\R^d}\tilde{R}_j^2\ddj udx ,
\end{aligned}
\end{equation}
as well as
\begin{equation}\label{4.1403}
\begin{aligned}
&\frac{1}{2}\frac{d}{dt}\|\ddj \theta\|_{L^2}^2 +\|\ddj\nabla \theta\|_{L^2}^2 + \int_{\R^d}\ddj\text{div}u\ddj \theta dx+\int_{\R^d}\ddj\text{div}(u\theta)\ddj \theta dx \\
&-\int_{\R^d}\frac{a}{1+a}\ddj\Delta \theta\ddj\theta dx = \int_{\R^d}\tilde{R}_j^3\ddj \theta dx.
\end{aligned}
\end{equation}
The combination of \eqref{4.1401}-\eqref{4.1403} leads to
\begin{equation}\label{4.1404}
\begin{aligned}
&\quad\frac{1}{2}\frac{d}{dt}\left(\int_{\R^d}\frac{1+\theta}{(1+a)^2}(\ddj a)^2 + (\ddj u)^2 + (\ddj\theta)^2dx  \right)+ \|\ddj u\|_{L^2}^2 + \|\ddj\nabla\theta\|_{L^2}^2 \\
&\leq \frac{1}{2}\Big\|\frac{\partial}{\partial t}\Big(\frac{1+\theta}{(1+a)^2}\Big)\Big\|_{L^\infty}\|\ddj a\|_{L^2}^2 + \Big\|\nabla\Big(\frac{1+\theta}{1+a}\Big)\Big\|_{L^\infty}\|\ddj u\|_{L^2}\|\ddj a\|_{L^2} \\
&\quad+ \frac{1}{2}\Big\|\text{div}\Big(\frac{1+\theta}{(1+a)^2}u\Big)\Big\|_{L^\infty}\|\ddj a\|^2_{L^2} + \frac{1}{2}\|\text{div}u\|_{L^\infty}\|\ddj u\|^2_{L^2}+\|\ddj (u\theta)\|_{L^2}\|\ddj \nabla\theta\|_{L^2} \\[1mm]
&\quad+ \|\nabla(\frac{a}{1+a})\|_{L^\infty}\|\ddj\nabla\theta\|_{L^2}\|\ddj\theta\|_{L^2}+ \|\frac{a}{1+a}\|_{L^\infty}\|\ddj\nabla\theta\|_{L^2}^2 \\[1mm]
&\quad+\|\tilde{R}^1_j\|_{L^2}\|\frac{1+\theta}{(1+a)^2}\|_{L^{\infty}}\|\ddj a\|_{L^2} + \|\tilde{R}^2_j\|_{L^2}\|\ddj u\|_{L^2} + \|\tilde{R}^3_j\|_{L^2}\|\ddj \theta\|_{L^2}.
\end{aligned}
\end{equation}
By virtue of \eqref{123456} and \eqref{4.1404}, the Bernstein inequality and the fact $2^{-j}\leq 2$, the following inequality
\begin{equation}\label{4.1405}
\begin{aligned}
&\quad\frac{d}{dt}\cE_{2,j}(t)+\cD_{2,j}(t) \\
&\quad\leq \frac{1}{2}\Big\|\frac{\partial}{\partial t}\Big(\frac{1+\theta}{(1+a)^2}\Big)\Big\|_{L^\infty}\|\ddj a\|_{L^2}^2 + \Big\|\nabla\Big(\frac{1+\theta}{1+a}\Big)\Big\|_{L^\infty}\|\ddj u\|_{L^2}\|\ddj a\|_{L^2} \\
&\quad\quad+ \frac{1}{2}\Big\|\text{div}\Big(\frac{1+\theta}{(1+a)^2}u\Big)\Big\|_{L^\infty}\|\ddj a\|^2_{L^2} + \frac{1}{2}\|\text{div}u\|_{L^\infty}\|\ddj u\|^2_{L^2}+\|\ddj (u\theta)\|_{L^2}\|\ddj \nabla\theta\|_{L^2} \\[1mm]
&\quad\quad+ \|\nabla(\frac{a}{1+a})\|_{L^\infty}\|\ddj\nabla\theta\|_{L^2}\|\ddj\theta\|_{L^2}+ \|\frac{a}{1+a}\|_{L^\infty}\|\ddj\nabla\theta\|_{L^2}^2 \\[1mm]
&\quad\quad+\|\tilde{R}^1_j\|_{L^2}\|\frac{1+\theta}{(1+a)^2}\|_{L^{\infty}}\|\ddj a\|_{L^2} + \|\tilde{R}^2_j\|_{L^2}\|\ddj u\|_{L^2} + \|\tilde{R}^3_j\|_{L^2}\|\ddj \theta\|_{L^2} \\[1mm]
&\quad\quad+\eta_{2}2^{-2j}\|\ddj\text{div}(au)\|_{L^2}\|\ddj\text{div}u\|_{L^2}+\eta_{2}2^{-2j}\|\ddj(u\cdot\nabla u)\|_{L^2}\|\ddj\nabla a\|_{L^2}\\
&\quad\quad+\eta_{2}2^{-2j}\|\ddj\big(\frac{\theta-a}{a+1}\nabla a\big)\|_{L^2}\|\ddj\nabla a\|_{L^2}
\end{aligned}
\end{equation}
holds, where $\cE_{2,j}(t)$ and $\cD_{2,j}(t)$ are defined by
\begin{eqnarray*}
\left\{
\begin{aligned}
\cE_{2,j}(t)&\triangleq\frac{1}{2}\left(\int_{\R^d}\frac{1+\theta}{(1+a)^2}(\ddj a)^2 + (\ddj u)^2 + (\ddj\theta)^2dx  \right)+ \eta_2 2^{-2j}\int_{\R^d}\ddj \nabla a\ddj udx,\\
\cD_{2,j}(t)&\triangleq\|\ddj u\|^2_{L^2} + \|\ddj\nabla\theta\|^2_{L^2}+ \eta_2 2^{-2j}\|\ddj \nabla a\|_{L^2}^2 - \eta_2 2^{-2j}\|\ddj\text{div}u\|^2_{L^2}\\[1mm]
&\quad
+\eta_2 2^{-2j}\int_{\R^d}\ddj u\ddj\nabla adx
+\eta_2 2^{-2j}\int_{\R^d}\ddj\nabla\theta\ddj\nabla adx
\end{aligned}
\right.
\end{eqnarray*}
with constant $\eta_2 \in(0,1)$ to be determined later. Since \eqref{bounded}, we see that
$$\int_{\R^d}\frac{1+\theta}{(1+a)^2}(\ddj a)^2dx\leq\|\frac{1+\theta}{(1+a)^2}\|_{L^\infty}\|\ddj a\|^2_{L^2} \sim \|\ddj a\|^2_{L^2}.$$
It is easy to verify for any $j\geq-1$, one can choose a sufficiently small constant $\eta_{2}\in(0,1)$ so that we 
\begin{equation}\label{8*}
\cE_{2,j}(t) \sim\|(\ddj a,\ddj u, \ddj\theta)\|^2_{L^2}
\end{equation}
and
\begin{equation}\label{9*}
\cD_{1,j}(t)\gtrsim \|(\ddj a, \ddj u)\|^2_{L^2}+2^{2j}\|\ddj \theta\|^2_{L^2}.
\end{equation}
By vitrue of \eqref{4.6}, \eqref{4.1404}, the Bernstein inequality and the fact $2^{-j}\leq2$, \eqref{high*} holds.
\end{proof}

\section{The proof of Theorem \ref{decay-2}.} \label{S2.2}
In this section, we show Theorem \ref{decay-2} on the optimal time-decay rates of the strong solution to the Cauchy problem \eqref{2.2} in the case that $\|(a_0, u_0,\theta_0)\|^{\ell}_{\dB^{-\sigma_{1}}_{2,\infty}}$ is bounded.
\subsection{The regularity evolution of negative Besov norms.} Now, we turn to bound the evolution of negative Besov norm, which is the main ingredient in deriving the Lyapunov-type inequality for energy norms.
\begin{lem}\label{LowEsta}
Let $(a,u,\theta)$ be the global solution to the Cauchy problem \eqref{2.2} given by Theorem {\rm\ref{GlobalExist}}. Then, under the assumptions of Theorem {\rm\ref{decay-2}}, the following inequality holds:
\begin{equation}\label{3*}
\begin{aligned}
\cX_{L,\sigma_{1}}(t)&\triangleq\|(a,u,\theta)\|^{\ell}_{\tL_t^\infty(\dB^{-\sigma_{1}}_{2,\infty})} + \| a\|^{\ell}_{\tL_t^2(\dB^{-\sigma_{1}+1}_{2,\infty})} + \|u\|^{\ell}_{\tL_t^2(\dB^{-\sigma_{1}}_{2,\infty})} +\| \theta\|^{\ell}_{\tL_t^2(\dB^{-\sigma_{1}+1}_{2,\infty})}\\
&\leq C\delta_{0},~~t>0,
\end{aligned}
\end{equation}
where $\delta_{0}$ is defined by \eqref{decay-1}, where $C$ is a constant independent ot $t$. 
\end{lem}
\begin{proof}
We recall that the Lyapunov type inequality \eqref{ELowL} holds as stated in Lemma \ref{LowL}. By integrating \eqref{ELowL} over $[0,t]$ and then taking the square root both sides of the resulting inequality yields
\begin{eqnarray}\label{4.8}
\begin{aligned}
&\|(\ddj a,\ddj u,\ddj\theta)\|_{L^\infty_t(L^2)} +2^{j} \|(\ddj a,\ddj\theta)\|_{L^2_t(L^2)} + \|\ddj u\|_{L^2_t(L^2)}\\[1mm]
&\lesssim \|(\ddj a_0,\ddj u_0,\ddj\theta_0)\|_{L^2} +\left\|\left(\ddj(au),\ddj(u\cdot\nabla u),\ddj(u\theta)\right)\right\|_{L^2_t(L^2)} \\[1mm]
&\quad+ \left\|\left(\ddj(\frac{\theta-a}{a+1}\nabla a),\ddj\big(\frac{a}{a+1}\nabla\theta\big)\right)\right\|_{L^2_t(L^2)} +\|\ddj\big(\nabla(\frac{a}{a+1})\nabla\theta\big)\|_{L^1_t(L^2)}
\end{aligned}
\end{eqnarray}
due to the low-frequency $j\leq0$. Multiplying \eqref{4.8} by $2^{-\sigma_{1}j}$ and taking the supremum on both $[0,t]$ and $j\leq0$, we get
\begin{eqnarray}\label{4.9}
\begin{aligned}
&\|(a,u,\theta)\|^{\ell}_{\tL_t^\infty(\dB^{-\sigma_{1}}_{2,\infty})} + \| (a,\theta)\|^{\ell}_{\tL_t^2(\dB^{-\sigma_{1}+1}_{2,\infty})} + \|u\|^{\ell}_{\tL_t^2(\dB^{-\sigma_{1}}_{2,\infty})} \\
&\quad\lesssim  \|(a_0,u_0,\theta_0)\|^{\ell}_{\dB^{-\sigma_{1}}_{2,\infty}}+\|(au, u\cdot\nabla u, u\theta) \|^{\ell}_{\tL_t^2(\dB^{-\sigma_{1}}_{2,\infty})} \\
&\quad\quad+\|(\frac{\theta-a}{a+1}\nabla a, \frac{a}{a+1}\nabla\theta)\|^{\ell}_{\tL_t^2(\dB^{-\sigma_{1}}_{2,\infty})}+ \|\nabla(\frac{a}{a+1})\nabla\theta\|^{\ell}_{\tL_t^1(\dB^{-\sigma_{1}}_{2,\infty})}.
\end{aligned}
\end{eqnarray}
In what follows, we focus on the nonlinear terms on the right-hand side of \eqref{4.9}. According to Proposition \ref{ParalinearizationTheorem} and Proposition \ref{ProductEstimates} with $s_{1}=\frac{d}{2}, s_{2}=-\sigma_{1}$, one can get
\begin{equation}\label{4.10}
\begin{aligned}
&\|(au, u\cdot\nabla u, u\theta) \|^{\ell}_{\tL_t^2(\dB^{-\sigma_{1}}_{2,\infty})}\\
&\lesssim \|(a,u,\theta)\|_{\tL_t^\infty(\dB^{-\sigma_{1}}_{2,\infty})}\|(u,\nabla u)\|_{\tL_t^2(\dB^{\frac{d}{2}}_{2,1})}  \\
&\lesssim\left(\cX_{L,\sigma_{1}}+\cX(t)\right)\cX(t),
\end{aligned}   
\end{equation}

\begin{equation}\label{***}
\begin{split}
&\|(\frac{\theta-a}{a+1}\nabla a, \frac{a}{a+1}\nabla\theta)\|^{\ell}_{\tL_t^2(\dB^{-\sigma_{1}}_{2,\infty})}\\
&\quad\lesssim\|(a,\theta)\|_{\tL_t^\infty(\dB^{\frac{d}{2}}_{2,1})}\| a\|_{\tL_t^2(\dB^{-\sigma_{1}+1}_{2,\infty})}+\|a\|_{\tL_t^\infty(\dB^{\frac{d}{2}}_{2,1})}\| \theta\|_{\tL_t^2(\dB^{-\sigma_{1}+1}_{2,\infty})}\\
&\quad\lesssim\left(\cX_{L,\sigma_{1}}+\cX(t)\right)\cX(t),
\end{split}
\end{equation}
and
\begin{equation}\label{4.12}
\begin{aligned}
\|\nabla(\frac{a}{a+1})\nabla\theta\|^{\ell}_{\tL_t^1(\dB^{-\sigma_{1}}_{2,\infty})}&\lesssim \|\frac{a}{a+1}\|_{\tL_t^2(\dB^{\frac{d}{2}+1}_{2,1})}\|\theta\|_{\tL_t^2(\dB^{-\sigma_{1}+1}_{2,\infty})}\\
&\lesssim \|a\|_{\tL_t^2(\dB^{\frac{d}{2}+1}_{2,1})}\|\theta\|_{\tL_t^2(\dB^{-\sigma_{1}+1}_{2,\infty})}\\
&\lesssim \cX(t)\left(\cX_{L,\sigma_{1}}+\cX(t)\right).
\end{aligned}
\end{equation}
Substituting the estimates \eqref{4.10}-\eqref{4.12} into \eqref{4.9}, we have
\begin{eqnarray*}
\cX_{L,\sigma_{1}}\lesssim  \|(a_0,u_0,\theta_0)\|^{\ell}_{\dB^{-\sigma_{1}}_{2,\infty}} + \cX(t)\left(\cX_{L,\sigma_{1}}+\cX(t)\right).
\end{eqnarray*}
Making use of \eqref{IniCon}, $\cX(t)\lesssim\cX(0)\ll1$ and $\|(a_0,u_0,\theta_0)\|^{\ell}_{\dB^{-\sigma_{1}}_{2,\infty}}+\cX(0)\sim\delta_{0}$, we prove \eqref{3*}. The proof of Lemma \ref{LowEsta} is completed.
\end{proof}

Next, we introduce a new time-weighted energy functional
\begin{equation}\label{**1}
\begin{split}
\cX_{M}&\triangleq\|(1+\tau)^{M}(a,u,\theta)\|^{\ell}_{\tL^\infty_t(\dB^{\frac{d}{2}}_{2,1})} +\|(1+\tau)^{M}(a,\theta)\|^{\ell}_{\tL^2_t(\dB^{\frac{d}{2}+1}_{2,1})}\\
&+\|(1+\tau)^{M}u\|^{\ell}_{\tL^2_t(\dB^{\frac{d}{2}}_{2,1})}+\|(1+\tau)^{M}(a,u,\theta)\|^{h}_{\tL^\infty_t(\dB^{\frac{d}{2}+1}_{2,1})}\\
&\quad+\|(1+\tau)^{M}(a,u)\|^{h}_{\tL^2_t(\dB^{\frac{d}{2}+1}_{2,1})} + \|(1+\tau)^{M}\theta\|^{h}_{\tL^2_t(\dB^{\frac{d}{2}+2}_{2,1})},
\end{split}
\end{equation}
where $M>0$ is chosen sufficiently large. Consequently, we have the following time-weighted Lyapunov
estimate.
\begin{lem}\label{LowEsta-1}
Let $(a,u,\theta)$ be the global solution to the Cauchy problem \eqref{2.2} given by Theorem {\rm\ref{GlobalExist}}. Under
the assumption of Theorem {\rm\ref{decay-2}}, it holds that
\begin{equation}\label{**2}
\cX_{M}\lesssim\delta_{0}(1+t)^{M-\frac{1}{2}(\frac{d}{2}+\sigma_{1})}
\end{equation}
for $M>1+\frac{1}{2}(\frac{d}{2}+\sigma_{1})$ and $t>0$, where $\delta_{0}\triangleq\|(a_0,u_0,\theta_0)\|^{\ell}_{\dB^{-\sigma_{1}}_{2,\infty}}+\|(a_0,u_0,\theta_0)\|^{h}_{\dB^{\frac{d}{2}+1}_{2,1}}.$
\end{lem}
\begin{proof}
The proof is separated into several steps.
\begin{itemize}
    \item \textit{Step1: Low-frequency estimates}
\end{itemize}
Let us begin with the Lyapunov type inequality \eqref{ELowL} in the low-frequency regime. Multiplying \eqref{ELowL} by $(1+t)^{2M}$ and using the fact 
\begin{align*}
\begin{split}
&(1+t)^{2M}\frac{d}{dt}\|(\ddj a,\ddj u,\ddj\theta)\|^{2}_{L^{2}}\\
&\quad=\frac{d}{dt}\|(1+t)^{M}(\ddj a,\ddj u,\ddj\theta)\|^{2}_{L^{2}}-2M(1+t)^{2M-1}\|(\ddj a,\ddj u,\ddj\theta)\|^{2}_{L^{2}},     
\end{split}
\end{align*}
we obtain
\begin{eqnarray*}
\begin{aligned}
&\frac{d}{dt}\|(1+\tau)^{M}(\ddj a,\ddj u, \ddj\theta)\|^2_{L^2} +(1+t)^{2M}\left( 2^{2j}\|(\ddj a, \ddj\theta)\|^2_{L^2}+\|\ddj u\|^2_{L^2}\right)\\
&\lesssim(1+t)^{2M-1}\|(\ddj a,\ddj u,\ddj\theta)\|^{2}_{L^{2}} \\
&\quad+ (1+t)^{2M}\|\left(\ddj(au),\ddj(u\cdot\nabla u),\ddj(\frac{\theta-a}{a+1}\nabla a)\right)\|_{L^2}\|(\ddj u,\ddj \nabla a)\|_{L^2} \\
&\quad+(1+t)^{2M} \|\left(\ddj(u\theta),\ddj\big(\frac{a}{a+1}\nabla\theta\big)\right)\|_{L^2}\|\ddj\nabla\theta\|_{L^2}\\
&\quad+(1+t)^{2M}\|\ddj\big(\nabla(\frac{a}{a+1})\nabla\theta\big)\|_{L^2}\|\ddj\theta\|_{L^2}
\end{aligned}
\end{eqnarray*}
for $j\leq0$, which implies that
\begin{eqnarray}\label{4.6}
\begin{aligned}
&\|(1+t)^{M}(\ddj a,\ddj u, \ddj\theta)\|_{L^2} + 2^{j}\|(1+\tau)^{M}(\ddj a, \ddj\theta)\|_{L^2_{t}L^2}+\|(1+\tau)^{M}\ddj u\|_{L^2_{t}L^2}\\
&\quad\lesssim\|(\ddj a_0,\ddj u_0, \ddj\theta_0)\|_{L^2}+\|(1+\tau)^{M-\frac{1}{2}}(\ddj a,\ddj u,\ddj\theta)\|_{L^{2}_{t}L^{2}}\\
&\quad\quad+\left\|(1+\tau)^{M}\left(\ddj(au),\ddj(u\cdot\nabla u),\ddj(u\theta),\ddj(\frac{\theta-a}{a+1}\nabla a),\ddj\big(\frac{a}{a+1}\nabla\theta\big)\right)\right\|_{L^2_t(L^2)} \\[1mm]
&\quad\quad +\|(1+\tau)^{M}\ddj\big(\nabla(\frac{a}{a+1})\nabla\theta\big)\|_{L^1_t(L^2)}.
\end{aligned}
\end{eqnarray}
Then, we multiply \eqref{4.6} by $2^{\frac{d}{2}j}$, take the supremum on $[0,t]$, and then sum over $j\leq0$ to get
\begin{eqnarray}\label{4.009*}
\begin{aligned}
&\|(1+\tau)^{M}(a,u,\theta)\|^{\ell}_{\tL_t^\infty(\dB^{\frac{d}{2}}_{2,1})} + \|(1+\tau)^{M} (a,\theta)\|^{\ell}_{\tL_t^2(\dB^{\frac{d}{2}+1}_{2,1})} + \|(1+\tau)^{M}u\|^{\ell}_{\tL_t^2(\dB^{\frac{d}{2}}_{2,1})}\\
&\lesssim  \|(a_0,u_0,\theta_0)\|^{\ell}_{\dB^{\frac{d}{2}}_{2,1}}+\|(1+\tau)^{M-\frac{1}{2}}(a,u,\theta)\|^{\ell}_{\tL_t^2(\dB^{\frac{d}{2}}_{2,1})}\\
&\quad+\|(1+\tau)^{M}(au,u\cdot\nabla u, u\theta,\frac{\theta-a}{a+1}\nabla a,\frac{a}{a+1}\nabla\theta)\|^{\ell}_{\tL_t^2(\dB^{\frac{d}{2}}_{2,1})}\\
&\quad+ \|(1+\tau)^{M}\nabla(\frac{a}{a+1})\nabla\theta\|^{\ell}_{\tL_t^1(\dB^{\frac{d}{2}}_{2,1})} .
\end{aligned}
\end{eqnarray}
It is worth emphasized that the second term on the right-hand side of \eqref{4.009*} plays a key role in the
derivation of decay rates. To bound it, it follows from Lemma \ref{interpolation-1} and Young’s inequality that
\begin{equation}\label{0001*}
\begin{split}
&\|(1+\tau)^{M-\frac{1}{2}}(a,u,\theta)\|^{\ell}_{\tL_t^2(\dB^{\frac{d}{2}}_{2,1})}\\
&\quad\leq C\|(1+\tau)^{M-1}(a,u,\theta)\|^{\ell}_{\tL_t^1(\dB^{\frac{d}{2}}_{2,1})}+\frac{1}{4}\|(1+\tau)^{M}(a,u,\theta)\|^{\ell}_{\tL_t^\infty(\dB^{\frac{d}{2}}_{2,1})},
\end{split}
\end{equation}
where the first term can be handled as follows:
\begin{equation}\label{0002*}
\begin{split}
&\|(1+\tau)^{M-1}(a,u,\theta)\|^{\ell}_{\tL_t^1(\dB^{\frac{d}{2}}_{2,1})}\\
&\quad\lesssim\int_{0}^{t}(1+\tau)^{M-1}\|(a,u,\theta)^{\ell}\|_{\dB^{\frac{d}{2}}_{2,1}}d\tau+\int_{0}^{t}(1+\tau)^{M-1}\|(a,u,\theta)^{h}\|_{\dB^{\frac{d}{2}}_{2,1}}d\tau.
\end{split}
\end{equation}
To control the first term on the right-hand side of \eqref{0002*}, we deduce from Lemma \ref{interpolation-1} with $s_{1}=-\sigma_{1}, s_{2}=\frac{d}{2}, p=2$ and $\eta_{0}=\frac{1}{\sigma_{1}+1+\frac{d}{2}}$ that
\begin{equation}\label{0003*}
\begin{split}
&\int_{0}^{t}(1+\tau)^{M-1}\|(a,u,\theta)^{\ell}\|_{\dB^{\frac{d}{2}}_{2,1}}d\tau\\
&\quad\lesssim\int_{0}^{t}(1+\tau)^{M-1}\|(a,u,\theta)^{\ell}\|^{\eta_{0}}_{\dB^{-\sigma_{1}}_{2,\infty}}\|(a,u,\theta)^{\ell}\|^{1-\eta_{0}}_{\dB^{\frac{d}{2}+1}_{2,\infty}}d\tau\\
&\quad\lesssim\|(a,u,\theta)^{\ell}\|^{\eta_{0}}_{\tL^{\infty}_{t}(\dB^{-\sigma_{1}}_{2,\infty})}\int_{0}^{t}\|(1+\tau)^{M}(a,u,\theta)^{\ell}\|^{1-\eta_{0}}_{\dB^{\frac{d}{2}+1}_{2,1}}(1+\tau)^{M\eta_{0}-1}d\tau\\
&\quad\lesssim\|(a,u,\theta)^{\ell}\|^{\eta_{0}}_{\tL^{\infty}_{t}(\dB^{-\sigma_{1}}_{2,\infty})}\left(\|(1+\tau)^{M}(a,u,\theta)^{\ell}\|_{\tL^{2}_{t}(\dB^{\frac{d}{2}+1}_{2,1})}\right)^{1-\eta_{0}}\|(1+\tau)^{M\eta_{0}-1}\|_{L^{\frac{2}{1+\eta_{0}}}_{t}}.
\end{split}
\end{equation}
Taking advantage of \eqref{E.q6.7} and the dissipative properties of $(a,u,\theta)$ for high frequencies, it is easy to verify that
\begin{equation}\label{0005*}
\begin{split}
&\int_{0}^{t}(1+t)^{M-1}\|(a,u)^{h}\|_{\dB^{\frac{d}{2}}_{2,1}}d\tau\\
&\quad\lesssim\left(\|(a,u)\|^{h}_{\tL^{\infty}_{t}(\dB^{\frac{d}{2}}_{2,1})}\right)^{\eta_{0}}\int_{0}^{t}\left(\|(1+\tau)^{M}(a,u)\|^{h}_{\dB^{\frac{d}{2}}_{2,1}}\right)^{1-\eta_{0}}(1+\tau)^{M\eta_{0}-1}d\tau\\
&\quad\lesssim\left(\|(1+\tau)^{M}(a,u)\|^{h}_{\tL^{2}_{t}(\dB^{\frac{d}{2}+1}_{2,1})}\right)^{1-\eta_{0}}\left(\|(a,u)\|^{h}_{\tL^{\infty}_{t}(\dB^{\frac{d}{2}+1}_{2,1})}\right)^{\eta_{0}}\|(1+\tau)^{M\eta_{0}-1}\|_{L^{\frac{2}{1+\eta_{0}}}_{t}},
\end{split}
\end{equation}
and
\begin{equation}\label{0005*-1}
\begin{split}
&\int_{0}^{t}(1+t)^{M-1}\|\theta^{h}\|_{\dB^{\frac{d}{2}}_{2,1}}d\tau\\
&\quad\lesssim\left(\|\theta\|^{h}_{\tL^{\infty}_{t}(\dB^{\frac{d}{2}}_{2,1})}\right)^{\eta_{0}}\int_{0}^{t}\left(\|(1+\tau)^{M}\theta\|^{h}_{\dB^{\frac{d}{2}}_{2,1}}\right)^{1-\eta_{0}}(1+\tau)^{M\eta_{0}-1}d\tau\\
&\quad\lesssim\left(\|(1+\tau)^{M}\theta\|^{h}_{\tL^{2}_{t}(\dB^{\frac{d}{2}+2}_{2,1})}\right)^{1-\eta_{0}}\left(\|\theta\|^{h}_{\tL^{\infty}_{t}(\dB^{\frac{d}{2}+1}_{2,1})}\right)^{\eta_{0}}\|(1+\tau)^{M\eta_{0}-1}\|_{L^{\frac{2}{1+\eta_{0}}}_{t}}.
\end{split}
\end{equation}
By combining \eqref{0001*}-\eqref{0005*-1}, Young's inequality and the fact that $$\|(1+\tau)^{M\eta_{0}-1}\|_{L^{\frac{2}{1+\eta_{0}}}_{t}}\lesssim\left((1+t)^{M-\frac{1}{2}(\frac{d}{2}+\sigma_{1})}\right)^{\eta_{0}},$$ we deduce
\begin{equation}\label{0006*}
\begin{split}
&\|(1+\tau)^{M-\frac{1}{2}}(a,u,\theta)\|^{\ell}_{\tL_t^2(\dB^{\frac{d}{2}}_{2,1})}\\
&\quad\leq\left(\|(a,u,\theta)^{\ell}\|_{\tL^{\infty}_{t}(\dB^{-\sigma_{1}}_{2,\infty})}+\|(a,u,\theta)^{h}\|_{\tL^{\infty}_{t}(\dB^{\frac{d}{2}+1}_{2,1})}\right)(1+t)^{M-\frac{1}{2}(\frac{d}{2}+\sigma_{1})}\\
&\quad\quad+\frac{1}{4}\left(\|(1+\tau)^{M}(a,\theta)\|^{\ell}_{\tL^{2}_{t}(\dB^{\frac{d}{2}+1}_{2,1})}+\|(1+\tau)^{M}u\|^{\ell}_{\tL^{2}_{t}(\dB^{\frac{d}{2}}_{2,1})}\right)\\
&\quad\quad+\frac{1}{4}\left(\|(1+\tau)^{M}(a,u)\|^{h}_{\tL^{2}_{t}(\dB^{\frac{d}{2}+1}_{2,1})}+\|(1+\tau)^{M}\theta\|^{h}_{\tL^{2}_{t}(\dB^{\frac{d}{2}+2}_{2,1})}\right).
\end{split}
\end{equation}
According to Proposition \ref{ParalinearizationTheorem} and Proposition \ref{ProductEstimates}, the nonlinearities on the right-hand side of \eqref{4.009*} can be estimated by
\begin{equation}\label{0007*}
\begin{split}
&\|(1+\tau)^{M}(au,u\cdot\nabla u, u\theta,\frac{\theta-a}{a+1}\nabla a,\frac{a}{a+1}\nabla\theta)\|^{\ell}_{\tL_t^2(\dB^{\frac{d}{2}}_{2,1})}\\
&\quad\lesssim \|(a,u,\theta)\|_{\tL_t^\infty(\dB^{\frac{d}{2}}_{2,1})}\|(1+\tau)^{M}(u,\nabla u)\|_{\tL_t^2(\dB^{\frac{d}{2}}_{2,1})}\\
&\quad\quad+ \|(a,\theta)\|_{\tL_t^\infty(\dB^{\frac{d}{2}}_{2,1})}\|(1+\tau)^{M} a\|_{\tL_t^2(\dB^{\frac{d}{2}+1}_{2,1})}+\|a\|_{\tL_t^\infty(\dB^{\frac{d}{2}}_{2,1})}\|(1+\tau)^{M} \theta\|_{\tL_t^2(\dB^{\frac{d}{2}+1}_{2,1})}\\
&\quad\lesssim\cX(t)\cX_{M}(t)
\end{split}
\end{equation}
and
\begin{equation}\label{008*}
\begin{split}
\|(1+\tau)^{M}\nabla(\frac{a}{a+1})\nabla\theta\|^{\ell}_{\tL_t^1(\dB^{\frac{d}{2}}_{2,1})}&\lesssim \|a\|_{\tL_t^2(\dB^{\frac{d}{2}+1}_{2,1})}\|(1+\tau)^{M}\theta\|_{\tL_t^2(\dB^{\frac{d}{2}+1}_{2,1})}\\
&\lesssim \cX(t)\cX_{M}(t).
\end{split}
\end{equation}
Substituting \eqref{0006*}-\eqref{008*} into \eqref{4.009*}, we get the low-frequency bound
\begin{eqnarray}\label{0009*}
\begin{aligned}
&\|(1+\tau)^{M}(a,u,\theta)\|^{\ell}_{\tL_t^\infty(\dB^{\frac{d}{2}}_{2,1})} + \|(1+\tau)^{M} (a,\theta)\|^{\ell}_{\tL_t^2(\dB^{\frac{d}{2}+1}_{2,1})} + \|(1+\tau)^{M}u\|^{\ell}_{\tL_t^2(\dB^{\frac{d}{2}}_{2,1})}\\
&\quad\lesssim  \|(a_0,u_0,\theta_0)\|^{\ell}_{\dB^{\frac{d}{2}}_{2,1}}\\
&\quad\quad+\left(\|(a,u,\theta)^{\ell}\|_{\tL^{\infty}_{t}(\dB^{-\sigma_{1}}_{2,\infty})}+\|(a,u,\theta)\|^{h}_{\tL^{\infty}_{t}(\dB^{\frac{d}{2}+1}_{2,1})}\right)(1+t)^{M-\frac{1}{2}(\frac{d}{2}+\sigma_{1})}+\cX(t)\cX_{M}(t).
\end{aligned}
\end{eqnarray}
\begin{itemize}
    \item \textit{Step2: High-frequency estimates}
\end{itemize}
For any $j\geq-1$, we show after multiplying the Lyapunov type inequality \eqref{high*} by $(1+t)^{2M}$ that
\begin{equation}\label{00010*}
\begin{aligned}
&(1+t)^{M}\|(\ddj a,\ddj u, \ddj\theta)\|_{L^2} +\|(1+\tau)^{M}(\ddj a, \ddj u)\|_{L^2_{t}L^2}+2^{j}\|(1+\tau)^{M}\ddj \theta\|_{L^2_{t}L^2}\\
&\quad\lesssim\|(\ddj a_0,\ddj u_0, \ddj\theta_0)\|_{L^2}+\|(1+\tau)^{M-\frac{1}{2}}(\ddj a,\ddj u,\ddj\theta)\|_{L^{2}_{t}L^{2}}\\
&\quad\quad+\left(\int_{0}^{t}(1+\tau)^{2M}J_{1}\, d\tau\right)^{1/2}+\left(\int_{0}^{t}(1+\tau)^{2M}J_{2}\,d\tau\right)^{1/2}\\
&\quad\quad+\left(\int_{0}^{t}(1+\tau)^{2M}J_{3}\, d\tau\right)^{1/2}+2^{-j}\left(\int_{0}^{t}(1+\tau)^{2M}J_{4}\,d\tau\right)^{1/2},
\end{aligned}
\end{equation}
where
\begin{equation*}
\begin{split}
&J_{1}\triangleq\left\|\frac{\partial}{\partial t}\Big(\frac{1+\theta}{(1+a)^2}\Big)\right\|_{L^\infty}\|\ddj a\|_{L^2}^2+\left\|\nabla\Big(\frac{1+\theta}{1+a}\Big)\right\|_{L^\infty}\|\ddj u\|_{L^2}\|\ddj a\|_{L^2}\\
&\quad+\left\|\text{div}\Big(\frac{1+\theta}{(1+a)^2}u\Big)\right\|_{L^\infty}\|\ddj a\|^2_{L^2}+\left\|\text{div}u\right\|_{L^\infty}\|\ddj u\|^2_{L^2},\\
&J_{2}\triangleq\|\left(\ddj (u\theta),\|\frac{a}{1+a}\|_{L^\infty}\|\ddj \nabla\theta\|_{L^2},\|\nabla(\frac{a}{1+a})\|_{L^\infty}\|\ddj\theta\|_{L^2}\right)\|_{L^2}\|\ddj \nabla\theta\|_{L^2},\\
&J_{3}\triangleq\|\tilde{R}^1_j\|_{L^2}\|\frac{1+\theta}{(1+a)^2}\|_{L^{\infty}}\|\ddj a\|_{L^2}+\|\tilde{R}^2_j\|_{L^2}\|\ddj u\|_{L^2}+\|\tilde{R}^3_j\|_{L^2}\|\ddj \theta\|_{L^2},\\
&J_{4}\triangleq\|\ddj\text{div}(au)\|_{L^2}\|\ddj\text{div}u\|_{L^2}+\|\Big(\ddj(u\cdot\nabla u),\ddj\big(\frac{\theta-a}{a+1}\nabla a\big)\Big)\|_{L^2}\|\ddj\nabla a\|_{L^2}.
\end{split}
\end{equation*}
Thence performing direct computations on the above inequality, we have
\begin{equation}\label{4.20}
\begin{aligned}
& \|(1+\tau)^{M}(a, u,\theta)\|^h_{\tL^\infty_t(\dB^{\frac{d}{2}+1}_{2,1})} + \|(1+\tau)^{M}(a, u)\|^h_{\tL^2_t(\dB^{\frac{d}{2}+1}_{2,1})}+ \|(1+\tau)^{M}\theta\|^h_{\tL^2_t(\dB^{\frac{d}{2}+2}_{2,1})} \\
&\leq \|(a_0, u_0,\theta_0)\|^h_{\dB^{\frac{d}{2}+1}_{2,1}}+\|(1+\tau)^{M-\frac{1}{2}}(a,u,\theta)\|^{h}_{\tL_t^2(\dB^{\frac{d}{2}+1}_{2,1})}+\Sum_{i=i}^{4}I_{i},
\end{aligned}
\end{equation}
where
\begin{equation*}
\begin{split}
&I_{1}\triangleq\left\|\frac{\partial}{\partial t}\Big(\frac{1+\theta}{(1+a)^2}\Big)\right\|_{L^2_t(L^\infty)}\|(1+\tau)^{M} a\|^h_{\tL^2_t(\dB^{\frac{d}{2}+1}_{2,1})}\\
&\quad+\left\|\nabla\Big(\frac{1+\theta}{1+a}\Big)\right\|_{L^\infty_t(L^\infty)}\|(1+\tau)^{M} u\|_{\tL^2_t(\dB^{\frac{d}{2}+1}_{2,1})}\\
&\quad+\left\|\text{div}\left(\frac{1+\theta}{(1+a)^2}u\right)\right\|_{L^2_t(L^\infty)}\|(1+\tau)^{M}a\|^h_{\tL^2_t(\dB^{\frac{d}{2}+1}_{2,1})}+\|\text{div}u\|_{L^2_t(L^\infty)}\|(1+\tau)^{M}u\|^h_{\tL^2_t(\dB^{\frac{d}{2}+1}_{2,1})},\\
&I_{2}\triangleq\| (1+\tau)^{M}u\theta\|^h_{\tL^2_t(\dB^{\frac{d}{2}+1}_{2,1})}+ \|\nabla(\frac{a}{1+a})\|_{L^2_t(L^\infty)}\|(1+\tau)^{M}\nabla\theta\|^h_{\tL^2_t(\dB^{\frac{d}{2}+1}_{2,1})} \\
&\quad+\|\frac{a}{1+a}\|_{L^\infty_t(L^\infty)}\|(1+\tau)^{M}\nabla\theta\|^h_{\tL^2_t(\dB^{\frac{d}{2}+1}_{2,1})},\\
\end{split}
\end{equation*}
\begin{equation*}
\begin{split}
&I_{3}\triangleq\sum_{j\geq-1}2^{(\frac{d}{2}+1)j}\left( \|(1+\tau)^{M}\tilde{R}^1_j\|_{L^1_t(L^2)} + \|(1+\tau)^{M}\tilde{R}^2_j\|_{L^1_t(L^2)}+\|(1+\tau)^{M}\tilde{R}^3_j\|_{L^1_t(L^2)}\right),\\
&I_{4}\triangleq\|(1+\tau)^{M}\Big(\text{div}(au),(u\cdot\nabla u),\big(\frac{\theta-a}{a+1}\nabla a\big)\Big)\|^h_{\tL^2_t(\dB^{\frac{d}{2}}_{2,1})}.
\end{split}
\end{equation*}
By similar arguments as used in \eqref{0001*}-\eqref{0006*}, one can get
\begin{equation}\label{****}
\begin{split}
&\|(1+\tau)^{M-\frac{1}{2}}(a,u,\theta)\|^{h}_{\tL_t^2(\dB^{\frac{d}{2}+1}_{2,1})}\\
&\quad\leq C\int_{0}^{t}\|(1+\tau)^{M-1}(a,u,\theta)\|^{h}_{\dB^{\frac{d}{2}+1}_{2,1}}d\tau+\frac{1}{4}\|(1+\tau)^{M}(a,u,\theta)\|^{h}_{\tL_t^\infty(\dB^{\frac{d}{2}+1}_{2,1})}\\
&\quad\leq C\|(a,u,\theta)\|^{h}_{\tL_t^\infty(\dB^{\frac{d}{2}+1}_{2,1})}(1+t)^{M-\frac{1}{2}(\frac{3}{2}+\sigma_{1})}+\frac{1}{4}\|(1+\tau)^{M}(a,u,\theta)\|^{h}_{\tL_t^2(\dB^{\frac{d}{2}+1}_{2,1})}.
\end{split}
\end{equation}
The right-hand side of \eqref{4.20} can be controlled below. As in \eqref{2.2}, one can show
\begin{align*}
\big\|\frac{\partial}{\partial t}\big(\frac{1+\theta}{(1+a)^2}\big)\big\|_{L^\infty}\lesssim (1+\|(a,\theta)\|_{L^\infty})\|(\partial_t a,\partial_t \theta)\|_{L^\infty},
\end{align*}
where
\begin{align*}
\partial_{t}a&=- u\cdot\nabla a- (1+a)\text{div}u,\\
\partial_t\theta &=-(1+\theta)\text{div} u-u\cdot\nabla \theta +\frac{1}{a+1}\Delta \theta.
\end{align*}
Furthermore, straightforward computations lead to 
\begin{align*}
&\|\partial_t a\|_{L^\infty} \lesssim \|u\|_{L^\infty}\|\nabla a\|_{L^\infty} + \big(1+\|a\|_{L^\infty}\big)\|\text{div}u\|_{L^\infty},\\[1mm]
&\|\partial_t\theta \|_{L^\infty}\lesssim \big(1+\|\theta\|_{L^\infty}\big)\|\text{div} u\|_{L^\infty} + \|u\|_{L^\infty}\|\nabla \theta\|_{L^\infty} +\big(1+ \|a\|_{L^\infty}\big)\|\Delta \theta\|_{L^\infty}.
\end{align*}
With the above preparation, due to \eqref{bounded} we have the following estimates:
\begin{eqnarray}\label{4.21}
\left\{
\begin{aligned}
&\big\|\frac{\partial}{\partial t}\big(\frac{1+\theta}{(1+a)^2}\big)\big\|_{L^2_t(L^\infty)}\\
&\quad\lesssim  \|u\|_{L^\infty_t(L^\infty)}\|\nabla a\|_{L^2_t(L^\infty)} + \left(1+\|a\|_{L^\infty_t(L^\infty)}+\|\theta\|_{L^\infty_t(L^\infty)}\right)\|\text{div}u\|_{L^2_t(L^\infty)} \\
&\quad\quad+\|u\|_{L^\infty_t(L^\infty)}\|\nabla\theta\|_{L^2_t(L^\infty)}+ \left(1+\|a\|_{L^\infty_t(L^\infty)}\right)\|\Delta\theta\|_{L^2_t(L^\infty)},\\
&\Big\|\nabla\Big(\frac{1+\theta}{1+a}\Big)\Big\|_{L^\infty_t(L^\infty)}\lesssim\|(\nabla a,\nabla\theta)\|_{L^\infty_t(L^\infty)},\\
&\Big\|\text{div}\Big(\frac{1+\theta}{(1+a)^2}u\Big)\Big\|_{L^2_t(L^\infty)}\\
&\quad\lesssim \Big\|\nabla\Big(\frac{1+\theta}{(1+a)^2}\Big)\Big\|_{L^\infty_t(L^\infty)}\|u\|_{L^2_t(L^\infty)} + \Big\|\frac{1+\theta}{(1+a)^2}\Big\|_{L^\infty_t(L^\infty)}\|\text{div}u\|_{L^2_t(L^\infty)}\\[1mm]
&\quad\lesssim\|(\nabla a,\nabla\theta)\|_{L^\infty_t(L^\infty)}\|u\|_{L^2_t(L^\infty)} + \|\text{div}u\|_{L^2_t(L^\infty)}.
\end{aligned}
\right.
\end{eqnarray}
Thence it follows by $\dB^{\frac{d}{2}}_{2,1}(\R^{d}_{x})\hookrightarrow L^{\infty}(\R^{d}_{x})$ and \eqref{E.q6.7} that
\begin{equation}\label{above}
\begin{split}
I_{1}&\lesssim\left(\cX(t)+\cX^{2}(t)\right)\cX_{M}(t).
\end{split}
\end{equation}
According to Proposition \ref{ParalinearizationTheorem} and Proposition \ref{ProductEstimates}, one can get
\begin{equation}\label{4}
\begin{split}
I_{2}&\lesssim\|u\|_{\tL^2_t(\dB^{\frac{d}{2}+1}_{2,1})}\|(1+\tau)^{M}\theta\|_{\tL^\infty_t(\dB^{\frac{d}{2}+1}_{2,1})}\\
&\quad+\big(\|a\|_{\tL^2_t(\dB^{\frac{d}{2}+1}_{2,1})}+\|a\|_{\tL^\infty_t(\dB^{\frac{d}{2}+1}_{2,1})}\big)\|(1+\tau)^{M}\theta\|_{\tL^2_t(\dB^{\frac{d}{2}+2}_{2,1})}\\
&\lesssim\cX(t)\cX_{M}(t).
\end{split}
\end{equation}
It follows from the commutator estimate in Proposition \ref{CommutatorEstimate} that
\begin{equation}\label{4.006}
\begin{aligned}
I_{3}&\lesssim \|(1+\tau)^{M}a\|_{\widetilde{L}^2_t(\dB^{\frac{d}{2}+1}_{2,1})}\|u\|_{\widetilde{L}^2_t(\dB^{\frac{d}{2}+1}_{2,1})}+\|(1+\tau)^{M}a\|_{\widetilde{L}^2_t(\dB^{\frac{d}{2}+1}_{2,1})}\|\theta\|_{\widetilde{L}^2_t(\dB^{\frac{d}{2}+2}_{2,1})},\\[1mm]
&\quad+ \|(1+\tau)^{M}u\|_{\widetilde{L}^2_t(\dB^{\frac{d}{2}+1}_{2,1})}\|u\|_{\widetilde{L}^2_t(\dB^{\frac{d}{2}+1}_{2,1})} + \big\|\frac{\theta-a}{1+a}\big\|_{\widetilde{L}^2_t(\dB^{\frac{d}{2}+1}_{2,1})}\|(1+\tau)^{M}a\|_{\widetilde{L}^2_t(\dB^{\frac{d}{2}+1}_{2,1})}\\[1mm]
&\lesssim\cX(t)\cX_{M}(t).\\[1mm]
\end{aligned}
\end{equation}
Due to Proposition \ref{ProductEstimates}, we get
\begin{equation}\label{10}
\begin{split}
I_{4}&\lesssim\|(a,u,\theta)\|_{\tL^\infty_t(\dB^{\frac{d}{2}+1}_{2,1})}\|(1+\tau)^{M}(a,u)\|_{\tL^2_t(\dB^{\frac{d}{2}+1}_{2,1})}\\
&\lesssim\cX(t)\cX_{M}(t).
\end{split}
\end{equation}
Hence, in view of \eqref{****}, \eqref{above}-\eqref{10}, it holds that
\begin{equation}\label{000101*}
\begin{split}
& \|(1+\tau)^{M}(a, u,\theta)\|^h_{\tL^\infty_t(\dB^{\frac{d}{2}+1}_{2,1})} + \|(1+\tau)^{M}(a, u)\|^h_{\tL^2_t(\dB^{\frac{d}{2}+1}_{2,1})}+ \|(1+\tau)^{M}\theta\|^h_{\tL^2_t(\dB^{\frac{d}{2}+2}_{2,1})} \\
&\lesssim \|(a_0, u_0,\theta_0)\|^h_{\dB^{\frac{d}{2}+1}_{2,1}}+\|(a,u,\theta)\|^{h}_{\tL_t^\infty(\dB^{\frac{d}{2}+1}_{2,1})}(1+t)^{M-\frac{1}{2}(\frac{3}{2}+\sigma_{1})}+\cX_{M}(t)(\cX(t)+\cX^{2}(t)).
\end{split}
\end{equation}
\begin{itemize}
    \item \textit{Step 3: The gain of time-weighted estimates}
\end{itemize}
By combining \eqref{0009*} and \eqref{000101*}, we conclude that
\begin{flalign*}
\begin{aligned}
\cX_{M}(t)&\lesssim  \|(a_0,u_0,\theta_0)\|^{\ell}_{\dB^{\frac{d}{2}}_{2,1}}+\|(a_0,u_0,\theta_0)\|^{h}_{\dB^{\frac{d}{2}+1}_{2,1}}\\
&\quad+\left(\|(a,u,\theta)^{\ell}\|_{\tL^{\infty}_{t}(\dB^{-\sigma_{1}}_{2,\infty})}+\|(a,u,\theta)\|^{h}_{\tL^{\infty}_{t}(\dB^{\frac{d}{2}+1}_{2,1})}\right)(1+t)^{M-\frac{1}{2}(\frac{d}{2}+\sigma_{1})}\\
&\quad+(\cX(t)+\cX^{2}(t))\cX_{M}(t).
\end{aligned}
\end{flalign*}
Together with the global existence result (Theorem \ref{GlobalExist}) that implies $\cX(t)\lesssim\varepsilon_{0}\ll 1$, Lemma \ref{LowEsta} and
$$\|(a,u,\theta)\|^{h}_{\tL^{\infty}_{t}(\dB^{\frac{d}{2}+1}_{2,1})}\lesssim\|(a_0,u_0,\theta_0)\|^{\ell}_{\dB^{\frac{d}{2}}_{2,1}}+\|(a_0,u_0,\theta_0)\|^{h}_{\dB^{\frac{d}{2}+1}_{2,1}}\lesssim\delta_0,$$
we end up with \eqref{**2}. Therefore, the proof of Lemma \ref{LowEsta-1} is finished.
\end{proof}
\textit{Proof of Theorem \ref{decay-2}:}
For any $M>1+\frac{1}{2}(\frac{d}{2}+\sigma_{1})>1$, we obtain by Lemma \ref{LowEsta-1}
\begin{equation}\label{00011*}
\cX_{M}\lesssim\delta_{0}(1+t)^{M-\frac{1}{2}(\frac{d}{2}+\sigma_{1})}
\end{equation}
for all $t>0$ and any suitably large $M$. Therefore, after dividing \eqref{00011*} by $(1+t)^{M}$, it is easy to get
\begin{equation}\label{00012*}
\begin{split}
\|(a,u,\theta)(t)\|_{\dB^{\frac{d}{2}}_{2,1}}&\lesssim\|(a,u,\theta)(t)\|^{\ell}_{\dB^{\frac{d}{2}}_{2,1}}+\|(a,u,\theta)(t)\|^{h}_{\dB^{\frac{d}{2}+1}_{2,1}}\\
&\lesssim\delta_{0}(1+t)^{-\frac{1}{2}(\frac{d}{2}+\sigma_{1})},\quad t\geq1.
\end{split}
\end{equation}
Then it follows from \eqref{3*}, \eqref{00012*}, and the interpolation inequality \eqref{Interpolation} that
\begin{equation}\label{00013*}
\begin{split}
\|(a,u,\theta)^{\ell}(t)\|_{\dB^{\sigma}_{2,1}}&\lesssim\|(a,u,\theta)^{\ell}(t)\|^{\frac{\frac{d}{2}-\sigma}{\frac{d}{2}+\sigma_{1}}}_{\dB^{-\sigma_{1}}_{2,\infty}}\|(a,u,\theta)^{\ell}(t)\|^{\frac{\sigma+\sigma_{1}}{\frac{d}{2}+\sigma_{1}}}_{\dB^{\frac{d}{2}}_{2,1}}\\
&\lesssim\delta_{0}(1+t)^{-\frac{1}{2}(\frac{d}{2}+\sigma_{1})},~~\sigma\in(-\sigma_{1},\frac{d}{2}).
\end{split}
\end{equation}
By \eqref{00012*}-\eqref{00013*}, the optimal time-decay estimates in \eqref{decay} hold.

In order to improve the decay for the damped component $u$. $\eqref{2.2}_{2}$ can be rewritten as
\begin{equation}\label{*001}
\partial_{t} u+  u=-\nabla a - \nabla\theta- u\cdot\nabla u - \frac{\theta-a}{a+1}\nabla a.   
\end{equation}
We take the low-frequency $\dB^{-\sigma_1}_{2,\infty}$-norm of \eqref{*001} and make use of \eqref{decay} to get
\begin{equation}\label{*002}
\begin{split}
\|u\|^{\ell}_{\dB^{-\sigma_1}_{2,\infty}}&\lesssim e^{-t}\|u_{0}\|^{\ell}_{\dB^{-\sigma_1}_{2,\infty}}+\int_{0}^{t}e^{-(t-\tau)}\big(\|(a,\theta)\|^{\ell}_{\dB^{-\sigma_1+1}_{2,\infty}}\\
&\quad+\|u\cdot\nabla u\|^{\ell}_{\dB^{-\sigma_1}_{2,\infty}}+\|\frac{\theta-a}{a+1}\nabla a\|^{\ell}_{\dB^{-\sigma_1}_{2,\infty}}\big)d\tau,\quad t>0.
\end{split}
\end{equation}
Owing to $-\sigma_1+1\leq\frac{d}{2}$ and \eqref{decay}, we get
\begin{equation}\label{*003}
\|(a,\theta)\|^{\ell}_{\dB^{-\sigma_1+1}_{2,\infty}}\lesssim\delta_{0}(1+t)^{-\frac{1}{2}}.
\end{equation}
It also holds by \eqref{decay}, \eqref{ProdEs2} with $s_{1}=\frac{d}{2}$, $s_{2}=-\sigma_{1}$, \eqref{IniCon} and Proposition \ref{ParalinearizationTheorem} that
\begin{equation}\label{*004}
\|u\cdot\nabla u\|^{\ell}_{\dB^{-\sigma_1}_{2,\infty}}\lesssim\|u\|_{\dB^{\frac{d}{2}}_{2,1}}\|u\|_{\dB^{-\sigma_1+1}_{2,\infty}}\lesssim\delta_{0}(1+t)^{-\frac{1}{2}},
\end{equation}
and
\begin{equation}\label{*005}
\|\frac{\theta-a}{a+1}\nabla a\|^{\ell}_{\dB^{-\sigma_1}_{2,\infty}}\lesssim\|(a,\theta)\|_{\dB^{\frac{d}{2}}_{2,1}}\|a\|_{\dB^{-\sigma_1+1}_{2,\infty}}\lesssim\delta_{0}(1+t)^{-\frac{1}{2}}.
\end{equation}
Inserting \eqref{*003}-\eqref{*005} into \eqref{*002} shows
$$\|u\|_{\dB^{-\sigma_1}_{2,\infty}}\lesssim\|u\|^{\ell}_{\dB^{-\sigma_1}_{2,\infty}}+\|u\|^{h}_{\dB^{\frac{d}{2}+1}_{2,1}}\lesssim e^{-t}+\int_{0}^{t}e^{-(t-\tau)}\delta_{0}(1+t)^{-\frac{1}{2}}d\tau\lesssim\delta_{0}(1+t)^{-\frac{1}{2}},$$
where one has used \eqref{00012*} and the estimate
\begin{flalign}
\begin{split}
&\int_{0}^{t}e^{-(t-\tau)}\delta_{0}(1+t)^{-\frac{1}{2}}d\tau\\
&\quad\lesssim e^{-\frac{1}{2}t}\int_{0}^{\frac{t}{2}}\delta_{0}(1+t)^{-\frac{1}{2}}d\tau+(1+\frac{1}{2}t)^{-\frac{1}{2}}\int_{\frac{t}{2}}^{t}e^{-(t-\tau)}d\tau\lesssim\delta_{0}(1+t)^{-\frac{1}{2}}.
\end{split}
\end{flalign}

Finally, we take the $\dB^{\sigma}_{2,1}$-norm of \eqref{*001} for any $\sigma\in(-\sigma_1,\frac{d}{2}-1]$ to have
\begin{equation}\label{decay-4}
\begin{split}
\|u\|^{\ell}_{\dB^{\sigma}_{2,1}}&\lesssim e^{-t}\|u_{0}\|^{\ell}_{\dB^{\sigma}_{2,1}}+\int_{0}^{t}e^{-(t-\tau)}\big(\|(a,\theta)\|^{\ell}_{\dB^{\sigma+1}_{2,1}}
+\|u\cdot\nabla u\|^{\ell}_{\dB^{\sigma}_{2,1}}+\|\frac{\theta-a}{a+1}\nabla a\|^{\ell}_{\dB^{\sigma}_{2,1}}\big)d\tau\\
&\lesssim\delta_{0}(1+t)^{-\frac{1}{2}(1+\sigma+\sigma_1)},
\end{split}
\end{equation}
where one has used $\|(a,u,\theta)(t)\|^{\ell}_{\dB^{\sigma+1}_{2,1}}\leq C\delta_{0}(1+t)^{-\frac{1}{2}(1+\sigma+\sigma_{1})}$ derived from \eqref{decay}. Therefore, the following decay rates holds:
\begin{equation}\label{decay-5}
\|u\|_{\dB^{\sigma}_{2,1}}\lesssim\|u\|^{\ell}_{\dB^{\sigma}_{2,1}}+\|u\|^{h}_{\dB^{\frac{d}{2}+1}_{2,1}}\lesssim\delta_{0}(1+t)^{-\frac{1}{2}(1+\sigma+\sigma_1)}.
\end{equation}
Hence, \eqref{decay-3} is followed by \eqref{decay-4}-\eqref{decay-5} directly. The proof of Theorem \ref{decay-2} is complete.

\section{Appendix }\setcounter{equation}{0}

\subsection{Littlewood-Paley decomposition and Besov space}
Let us briefly review the definition of Besov spaces based on the Littlewood-Paley decomposition. The interested reader is referred to Chapter 2 and Chapter 3 of \cite{Bahouri-2011} for more details. Firstly, let's introduce the homogeneous Littlewood-Paley decomposition. For that purpose, we fix some smooth radial non increasing function $\chi $ with $\mathrm{Supp}\,\chi \subset
B\left(0,\frac {4}{3}\right)$ and $\chi \equiv 1$ on $B\left(0,\frac{3}{4}\right)$, then set $\varphi (\xi) =\chi (\xi/2)-\chi (\xi)$ so that
$$\sum_{j\in \mathbb{Z}}\varphi (2^{-j}\cdot ) =1,\ \ \
\mathrm{Supp}\,\varphi \subset \left\{\xi \in \mathbb{R}^{d}:\frac{3}{4}\leq |\xi|\leq \frac{8}{3}\right\}.$$
For any $j\in\Z$, define the homogeneous dyadic blocks $\dot{\Delta}_{j}$ by
$$\dot{\Delta}_{j}f\triangleq \varphi (2^{-j}D)f=\mathcal{F}^{-1}(\varphi
(2^{-j}\cdot )\mathcal{F}f)=2^{jd}h(2^{j}\cdot )\ast f\ \ \hbox{with}\ \
h\triangleq \mathcal{F}^{-1}\varphi,$$
where $\mathcal{F}$ and $\mathcal{F}^{-1}$ are the Fourier transform and its inverse.
The following Littlewood-Paley decomposition of $f$:
\begin{equation} \label{Eq:2.1}
f=\sum_{j\in \mathbb{Z}}\dot{\Delta}_{j}f,
\end{equation}
holds true modulo polynomials for any tempered distribution $f$. In order to have equality in the sense of tempered distributions, we consider only elements of the set $\mathcal{S}^{\prime }(\mathbb{R}^{d})$ of tempered distributions $f$ such that
\begin{equation}\label{Eq:2.2}
\lim_{j\rightarrow -\infty }\| \dot{S}_{j}f\| _{L^{\infty} }=0,
\end{equation}
where $\dot{S}_{j}f$ stands for the low frequency cut-off defined by $\dot{S}_{j}f\triangleq\chi (2^{-j}D)f$. Indeed, if \eqref{Eq:2.2} is fulfilled, then \eqref{Eq:2.1} holds in $\mathcal{S}'(\mathbb{R}^{d})$. For convenience, we denote by $\mathcal{S}'_{h}(\mathbb{R}^{d})$ the subspace of tempered distributions satisfying \eqref{Eq:2.2}.

Based on those dyadic blocks, Besov spaces are defined as follows.
\begin{defn}\label{D5.1}
For $s\in \R$ and $1\leq p,r\leq \infty$, the homogeneous Besov spaces $\dot{B}^{s}_{p,r}$ is defined by
$$\dot{B}^{s}_{p,r}\triangleq\left\{f\in \mathcal{S}'_{h}:\|f\|_{\dot{B}^{s}_{p,r}}<+\infty\right\},$$
where
\begin{align}\label{DefinBesov}
\|f\|_{\dB_{p,r}^{s}} \triangleq\| \{2^{js}\|\ddj f\|_{L^p}\}_{j\in\Z}\|_{l^{r}(\Z)}.
\end{align}
\end{defn}
The mixed space-time Besov spaces are also used, which was introduced by J. Y. Chemin and N. Lerner \cite{Chemin-1995}.
\begin{defn}\label{D5.2}
For $T>0, \, s\in\mathbb{R}$, $1\leq r,\,\varrho\leq\infty$, the homogeneous Chemin-Lerner space $\widetilde{L}^{\varrho}_{T}(\dot{B}^{s}_{p,r})$
is defined by
$$\widetilde{L}^{\varrho}_{T}(\dot{B}^{s}_{p,r})\triangleq\left\{f\in L^{\theta}(0,T;\mathcal{S}'_{h}):\|f\|_{\widetilde{L}^{\varrho}_{T}(\dot{B}^{s}_{p,r})}<+\infty\right\},$$
where
\begin{equation}\label{DefinLBesov}
\|f\|_{\widetilde{L}^{\varrho}_{T}(\dot{B}^{s}_{p,r})}\triangleq\|\{2^{js}\|\dot{\Delta}_{j} f\|_{L^{\varrho}_{T}(L^{p})}\}_{j\in\Z}\|_{l^{r}(\mathbb{Z})}.
\end{equation}
\end{defn}

For notational simplicity, index $T$ will be omitted if $T=+\infty $. We also use the following functional space:
\begin{equation*}
\cC_{b}(\mathbb{R_{+}};\dot{B}_{p,r}^{s})\triangleq \left\{f \in
\mathcal{C}(\mathbb{R_{+}};\dot{B}_{p,r}^{s}) \ | \ \|f\| _{\widetilde{L}^{\infty}(\R_+; \dot{B}_{p,r}^{s})}<+\infty \right\} .
\end{equation*}
The above norm
\eqref{DefinLBesov} may be linked with those of the standard spaces $L_{T}^{\theta} (\dot{B}_{p,r}^{s})$ by means of Minkowski's
inequality.
\begin{rem}\label{Rem2.1}
It holds that
$$\|f\|_{\widetilde{L}^{\varrho}_{T}(\dot{B}^{s}_{p,r})}\leq\|f\|_{L^{\varrho}_{T}(\dot{B}^{s}_{p,r})}\ \
{\rm if} \ \ r\geq\varrho; \ \ \ \
\|f\|_{\widetilde{L}^{\varrho}_{T}(\dot{B}^{s}_{p,r})}\geq\|f\|_{L^{\varrho}_{T}(\dot{B}^{s}_{p,r})}\ \
{\rm if}\ \  r\leq\varrho.
$$
\end{rem}
Restricting the norms in \eqref{DefinBesov} and \eqref{DefinLBesov} to the low frequency part and high frequency part of distributions will be fundamental to our approach. For example \cite{LiShou2023}, we fix some integer $j_{0}$ (the value of which will follow from the proofs of our main results) and put\footnote{Note that for technical reasons, we need to a small overlap between low and high frequencies.}

\begin{equation*}
\|f\| _{\dot{B}_{p,r}^{s}}^{\ell} \triangleq \|\{2^{js}\|\ddj f\|_{L^{p}}\}_{j \leq j_0}\|_{l^r} \ \hbox{and} \
\|f\|_{\dot{B}_{p,r}^{s}}^{h}\triangleq \|\{2^{js}\| \ddj f\|_{L^{p}}\}_{j \geq j_0-1}\|_{l^r},
\end{equation*}
\begin{equation*} 
\|f\|_{\widetilde{L}_{T}^{\varrho} (\dB_{p,r}^{s})}^{\ell} \triangleq \|\{2^{js}\|\ddj f\|_{L_{T}^{\varrho}(L^{p})}\}_{j\leq j_0}\|_{l^r}\ \hbox{and} \ \|f\| _{\widetilde{L}_{T}^{\varrho} (\dB_{p,r}^{s})}^{h}\triangleq \|\{2^{js}\|\ddj f\|_{L_{T}^{\varrho}(L^{p})}\}_{j \geq j_0-1}\|_{l^r}.
\end{equation*}
Define
\begin{align*}
f^{\ell}:=\sum_{j\leq-1}\ddj f, \ \ \ \  f^h:=f-f^{\ell} = \sum_{j\geq0}\ddj f.
\end{align*}
It is easy to check for any $s'>0$ that
\begin{align}\label{E.q6.7}
\begin{cases}
&\|f^{\ell}\|_{\dB_{p, r}^s} \lesssim\|f\|_{\dB_{p, r}^s}^{\ell} \lesssim\|f\|_{\dB_{p, r}^{s-s^{\prime}}}^{\ell},\\
&\|f^h\|_{\dB_{p, 1}^s} \lesssim\|f\|_{\dB_{p, r}^s}^h \lesssim\|f\|_{\dB_{p, r}^{s+s^{\prime}}}^h,\\ &\|f^{\ell}\|_{\widetilde{L}_T^{\varrho}(\dB_{p, r}^s)} \lesssim\|f\|_{\widetilde{L}_T^{\varrho}(\dot{B}_{p, r}^s)}^{\ell} \lesssim\|f\|_{\widetilde{L}_T^{\varrho}(\dB_{p, r}^{s-s^{\prime}})}^{\ell},\\ &\|f^h\|_{\widetilde{L}_T^{\varrho}(\dB_{p, r}^s)} \lesssim\|f\|_{\widetilde{L}_T^{\varrho}(\dB_{p, r}^s)}^h \lesssim\|f\|_{\widetilde{L}_T^{\varrho}(\dB_{p, r}^{s+s^{\prime}})}^h.
\end{cases}
\end{align}

We also need hybrid Besov spaces for which regularity assumptions are different in low frequencies and high frequencies \cite{Danchin2000}. We are going to recall the definition and properties.
\begin{defn}\label{def1}
Let $s,t\in\R.$ We define
\begin{align*}
\|f\|_{\dB^{s,t}_{2,1}} = \sum_{j\leq j_0}2^{js}\|\ddj f\|_{L^2} + \sum_{j >  j_0}2^{jt}\|\ddj f\|_{L^2}.
\end{align*}
Let $m = -[\frac{d}{2}+1-s],$ we then define
\begin{align*}
&\dB^{s,t}_{2,1}(\R^d) = \{f\in \mathcal{S'}(\R^d): \|f\|_{\dB^{s,t}_{2,1}} <\infty\}, \ \ \text{if}\ \  m<0,\\
&\dB^{s,t}_{2,1}(\R^d) = \{f\in \mathcal{S'}(\R^d)/\mathcal{P}_{m}: \|f\|_{\dB^{s,t}_{2,1}} <\infty\}, \ \ \text{if}\ \  m \ge 0.
\end{align*}
\end{defn}
\begin{rem} We have the following properties.
\begin{itemize}
\item $\dB^{s,s}_{2,1} = \dB^{s}_{2,1}.$
\item if $s\leq t$ then $\dB^{s,t}_{2,1} = \dB^{s}_{2,1} \cap \dB^{t}_{2,1}.$ Otherwise, $\dB^{s,t}_{2,1} = \dB^{s}_{2,1} + \dB^{t}_{2,1}.$
\item if $s_1\leq s_2$ and $t_1 \geq t_2,$ then $\dB^{s_1,t_1}_{2,1} \hookrightarrow \dB^{s_2,t_2}_{2,1}.$
\end{itemize}
\end{rem}

We recall some basic properties of Besov spaces and product estimates which will be used repeatedly in this paper. The first lemma is the so-called Bernstein inequalities.

\begin{lem}\label{Bernstein}
Let $k\in\N$, $1\leq a\leq b\leq\infty$, C is a constant and $f$ is an any function in $L^p$, then we have if $\mathrm{Supp}\, \mathcal{F}f\subset\left\{\xi\in \R^{d}: |\xi|\leq R\lambda\right\}$ for some $R>0$
\begin{equation*}
\|D^{k}f\|_{L^{b}}
\leq C^{1+k} \lambda^{k+d(\frac{1}{a}-\frac{1}{b})}\|f\|_{L^{a}}.
\end{equation*}
More generally, If \  $\mathrm{Supp}\, \mathcal{F}f\subset \left\{\xi\in \mathbb{R}^{d}: R_{1}\lambda\leq|\xi|\leq R_{2}\lambda\right\}$ for some $0<R_{1}<R_{2}$, we have
\begin{align*}
C^{-k-1}\lambda^k \|u\|_{L^a}\leq \|D^k u\|_{L^a}\leq C^{k+1} \lambda^{k} \|u\|_{L^a}
\end{align*}
\end{lem}

Due to the Bernstein inequalities, the Besov spaces have many properties:
\begin{lem}\label{interpolation-1}
Let $1\leq p,r,r_{1},r_{2}\leq\infty$.
\begin{itemize}
\item  \emph{Completeness:} $\dot{B}^{s}_{p,r}$ is a Banach space whenever $s<\frac{d}{p}$ or $s\leq \frac{d}{p}$ and $r=1$.

\item  \emph{Embedding:} For any $s\in \R, 1\leq p_1 \leq p_2 \leq \infty,$ and $1\leq r_1 \leq r_2 \leq \infty,$ it holds that
$$\dB^s_{p_1, r_1}\hookrightarrow \dB^{s-d(\frac{1}{p_1}-\frac{1}{p_2})}_{p_2, r_2}.$$
Moreover, For any $1\leq p \leq q \leq \infty,$ we have the continuous embedding
$$\dB^0_{p,1} \hookrightarrow L^p \hookrightarrow \dB^0_{p,\infty} \hookrightarrow \dB^\sigma_{q,\infty} \ \text{for}\  \sigma = -d(\frac{1}{p}-\frac{1}{q})<0.$$

\item  \emph{Interpolation:} The following inequalities are satisfied for $1\leq p,\,r_{1},\,r_{2},\,r\leq \infty$, $s_{1}< s_{2}$ and $\theta \in (0,1)$:
$$\|f\|_{\dot{B}_{p,r}^{\theta s_{1}+(1-\theta)s_{2}}}\lesssim \|f\| _{\dot{B}_{p,r_{1}}^{s_{1}}}^{\theta} \|f\|_{\dot{B}_{p,r_2}^{s_{2}}}^{1-\theta} \ \ \ \hbox{with} \ \ \ \frac{1}{r}=\frac{\theta}{r_{1}}+\frac{1-\theta}{r_{2}}.$$
Particularly, we have the following optimal interpolation formula
\begin{align}\label{Interpolation}
\|f\|_{\dB_{p,1}^{\theta s_1+(1-\theta)s_2}}\leq \frac{C}{\theta(1-\theta)(s_2-s_1)} \|f\|_{\dB_{p,\infty}^{s_1}}^{\theta}\|f\|_{\dB_{p,\infty}^{s_2}}^{1-\theta}.
\end{align}
\end{itemize}
\end{lem}
System \eqref{2.2} also involves multivariate compositions of functions (through $\frac{\theta-a}{a+1}$, $\frac{1}{a+1}$ that are bounded thanks to the following result:
\begin{prop}\label{ParalinearizationTheorem}{\rm(\cite{Xu-2023})}
Let $m\in\N$ and $s>0$. Let $G$ be a function in $\cC^{\infty}(\R^{m}\times\R^{d})$ such that $G(0,\cdots,0)=0$. Then for every real valued functions $f_{1},\cdots,f_{m}\in \dB_{p,r}^s\cap L^\infty$, the function $G(f_{1},\cdots,f_{m})$ belons to $\dB_{p,r}^s\cap L^\infty$ and we have
\begin{equation*}
\|G(f_{1},\cdots,f_{m})\|_{\dB_{p,r}^{s}}\leq C \|(f_{1},\cdots,f_{m})\|_{\dB_{p,r}^{s}}
\end{equation*}
with $C$ depending only on $\|f_{i}\|_{L^\infty}~(i=1,\cdots,m),\ G_{f_{i}}'\ $ (and higher derivatives), $s,\ p\ $ and $d$.

In the case $s>-\min(\frac{d}{p},\frac{d}{p*})$, then $f_{1},\cdots,f_{m}\in \dB_{p,r}^s\cap \dB_{p,1}^\frac{d}{p}$ implies that $G(f_{1},\cdots,f_{m})\in\dB_{p,r}^s\cap \dB_{p,1}^\frac{d}{p}$ and we have
\begin{equation*}
\|G(f_{1},\cdots,f_{m})\|_{\dB_{p,r}^{s}}\leq C \left(1+\|f_{1}\|_{\dB_{p,1}^{\frac{d}{p}}}+\cdots+\|f_{m}\|_{\dB_{p,1}^{\frac{d}{p}}}\right)\|(f_{1},\cdots,f_{m})\|_{\dB_{p,r}^{s}}.
\end{equation*}
\end{prop}

The following product estimates in Besov spaces play a fundamental role in our analysis of the nonlinear terms.
\begin{prop}\label{ProductEstimates}
The following statements holds:
\begin{itemize}
    \item Let $s>0$ and $1\leq p,r \leq  \infty $. Then $\dB_{p,r}^{s}\cap L^\infty$ is an algebra and
\begin{align}\label{ProdEs0}
\|fg\|_{\dB_{p,r}^{s}}\lesssim \|f\|_{L^\infty}  \|g\|_{\dB_{p,r}^{s}} + \|g\|_{L^\infty}  \|f\|_{\dB_{p,r}^{s}}.
\end{align}
\end{itemize}
\begin{itemize}
    \item Let the real numbers $s_1,\ s_2$ and $p$ satisfy $2\leq p\leq \infty,\ s_1\leq \frac{d}{p}, \ s_2\leq \frac{d}{p}$ and $s_1+s_2>0$. Then we have
\begin{align}\label{ProdEs1}
\|fg\|_{\dB_{p,1}^{s_1+s_2-\frac{d}{p}}}\lesssim \|f\|_{\dB_{p,1}^{s_1}}\|g\|_{\dB_{p,1}^{s_2}}.
\end{align}
\end{itemize}
\begin{itemize}
    \item Assume that $s_1,\ s_2$ and $p$ satisfy $2\leq p\leq \infty,\ s_1\leq \frac{d}{p}, \ s_2< \frac{d}{p}$ and $s_1+s_2\geq0.$ Then it holds that
\begin{align}\label{ProdEs2}
\|fg\|_{\dB_{p,\infty}^{s_1+s_2-\frac{d}{p}}}\lesssim \|f\|_{\dB_{p,1}^{s_1}}\|g\|_{\dB_{p,\infty}^{s_2}}.
\end{align}
\end{itemize}
\end{prop}
Finally, the following commutator estimates will be useful to control the nonlinearities in high frequencies.
\begin{prop}\label{CommutatorEstimate}
Let $1< p<\infty$, $1\leq \varrho\leq \infty$ and $s\in (-\frac{d}{p}-1,\frac{d}{p}]$. Then there exists a generic constant $C>0$ depending only on the dimension $d$ and the regular index $s$
\begin{equation}
\left\{\begin{array}{l}
\|[\ddj, f]g\|_{L^p}\leq C c_j 2^{-j(s+1)}\|f\|_{\dB_{p,1}^{\frac{d}{p}+1}}\|g\|_{\dB_{p,1}^{s}},\\
\|[\ddj, f]g\|_{L^\varrho_t(L^p)}\leq C c_j 2^{-j(s+1)}\|f\|_{\tL^{\varrho_1}_t(\dB_{p,1}^{\frac{d}{p}+1})}\|g\|_{\tL^{\varrho_2}_t(\dB_{p,1}^{s})},
\end{array}\right.
\end{equation}
with $\frac{1}{\varrho} = \frac{1}{\varrho_1} + \frac{1}{\varrho_2}$ and the commutator $[A,B]\triangleq AB-BA$.
\end{prop}

\end{document}